\documentclass{amsart}
\usepackage[numeric]{amsrefs}
\usepackage{amsmath,amssymb,amsthm,amsfonts,mathdots,amscd,bbm}
\usepackage{colordvi}

\newcommand{\tpoint}[1]{\vspace{3mm}\par \noindent \refstepcounter{subsubsection}{\bf \thesubsubsection.}
  {\textbf{#1} }}

\newcommand{\spoint}{\vspace{3mm}\par \noindent \refstepcounter{subsection}{\bf \thesubsection.} }

\newcommand{\aspoint}{\vspace{3mm}\par \noindent \refstepcounter{subsubsection}{\bf \thesubsubsection.} }

\newtheorem*{nthm}{Theorem}
\newtheorem*{nlem}{Lemma}
\newtheorem*{nprop}{Proposition}
\newtheorem*{ncor}{Corollary}

\theoremstyle{remark}

\numberwithin{equation}{section}

\newcommand{\mf}[1]{\mathfrak{#1}}
\newcommand{\hf}[1]{\widehat{\mathfrak{{#1}}}}
\newcommand{\wh}[1]{\widehat{#1}}
\newcommand{\mc}[1]{\mathcal{{#1}}}

\newcommand{\rr}{\rightarrow}
\renewcommand{\t}[1]{\widetilde{#1}}
\newcommand{\la}{\langle}
\newcommand{\ra}{\rangle}
\newcommand{\be}[1]{\begin{eqnarray} \label{#1}}
\newcommand{\ee}{\end{eqnarray}}
\newcommand{\vphi}{\phi}

\newcommand{\Z}{{\mathbb{Z}}}
\newcommand{\R}{{\mathbb{R}}}
\newcommand{\C}{{\mathbb{C}}}
\newcommand{\A}{{\mathbb{A}}}
\newcommand{\Q}{{\mathbb{Q}}}
\newcommand{\I}{{\mathbb{I}}}
\newcommand{\ad}{\mathbb{A}}

\newcommand{\hG}{\wh{G}}
\newcommand{\zee}{\mathbb{Z}}
\newcommand{\hA}{\wh{A}}
\newcommand{\hU}{\wh{U}}
\newcommand{\hK}{\wh{K}}
\newcommand{\hgam}{\wh{\Gamma}}
\newcommand{\hp}{\widehat{P}}
\newcommand{\hg}{\widehat{G}}
\newcommand{\ha}{\widehat{A}}
\newcommand{\hh}{\widehat{H}}

\newcommand{\hdelta}{\widehat{\Delta}}

\newcommand{\hu}{\widehat{U}}
\newcommand{\hb}{\widehat{B}}
\newcommand{\hk}{\widehat{K}}

\newcommand{\hP}{\wh{P}}
\renewcommand{\th}{\theta}
\newcommand{\aw}{\wh{W}}
\newcommand{\hB}{\wh{B}}
 
\newcommand{\hs}{\wh{\mathfrak{S}}}
\newcommand{\hed}{(\hf{h}^e)^*}
\newcommand{\he}{\hf{h}^e}

\newcommand{\iwa}{\iw_{\ha}}
\newcommand{\iwea}{\iw_{\eta(s) \ha}\,}
\newcommand{\iwam}{\iw_{M_\th}}
\newcommand{\lal}{\Lambda_{\ell+1}}
\newcommand{\car}{\wh{\bf{C}}}
\newcommand{\hPi}{\wh{\Pi}}
\newcommand{\hQ}{\wh{Q}}
\newcommand{\hLambda}{\wh{\Lambda}}
\newcommand{\dd}{\mathbf{D}}
\DeclareMathOperator{\Lie}{Lie}
\DeclareMathOperator{\Ad}{Ad}

\newcommand{\gad}{\hg_{\A}}
\newcommand{\kad}{\hk_{\A}}
\newcommand{\had}{\hh_{\A}}
\newcommand{\uad}{\hu_{\A}}
\newcommand{\bad}{\hb_{\A}}
\newcommand{\aiw}{\iw^{\A}}
\newcommand{\aiwa}{\iw^{\A}_{\ha}}
\newcommand{\aiwea}{\iw^{\A}_{\eta(s) \ha}\,}

\DeclareMathOperator{\iw}{Iw}

\renewcommand\a{\alpha}
\renewcommand\b{\beta}
\newcommand\g{\gamma}
\renewcommand\d{\delta}

\renewcommand\l{\lambda}
\renewcommand\L{\Lambda}
\newcommand\D{\Delta}

\renewcommand\D{\Delta}
\newcommand\G{\Gamma}
\newcommand\f{\frac}
\newcommand\smallf[2]{{\textstyle{\frac{#1}{#2}}}}

\newcommand\re{\text{Re~}}

\renewcommand\Re{\text{Re~}}

\newcommand{\Aut}{\operatorname{Aut}}

\renewcommand\i{^{-1}}
\renewcommand\({\left(}
\renewcommand\){\right)}
\newcommand{\ttwo}[4]{
\(\begin{smallmatrix}{#1} & {#2}
\\ {#3} & {#4} \end{smallmatrix}\)}

\newcommand{\sgn}{\operatorname{sgn}}

\newcommand\srel[2]{\begin{smallmatrix} {#1} \\ {#2} \end{smallmatrix}}

\newcommand{\mcen}{\iw_{\ha_{cen}}}

\makeatletter
\def\imod#1{\allowbreak\mkern5mu({\operator@font mod}\,#1)}
\makeatother

\begin{document}

\title{Entirety of cuspidal Eisenstein series on loop groups}

\author{Howard Garland}
\address{Department of Mathematics, Yale University, PO Box 208283,
New Haven, CT 06520-8283}\email{garland-howard@yale.edu}

\author{Stephen D. Miller}
\address{Department of Mathematics, Rutgers University, 110 Frelinghuysen Road, Piscataway, NJ 08544}\email{miller@math.rutgers.edu}
\thanks{Supported by NSF grant DMS-1201362.}

\author{Manish M. Patnaik}
\address{Department of Mathematical and Statistical Sciences, University of Alberta,
Edmonton, Alberta  T6G 2G1
CANADA}\email{patnaik@ualberta.ca}\thanks{Supported in part by an NSF Postdoctoral Fellowship DMS-0802940 and an University of Alberta Startup grant}

\date{January 5, 2016}

\maketitle

\begin{abstract} In this paper, we prove the entirety of  loop group Eisenstein series   induced from cusp forms on the underlying finite dimensional group, by demonstrating their absolute convergence on the full complex plane.
This is quite in contrast to the finite-dimensional setting, where such series only converge absolutely in a right half plane (and  have poles elsewhere coming from $L$-functions in their constant terms).
Our result is the $\Q$-analog of a theorem of A.~Braverman and D.~Kazhdan from the function field setting  \cite[Theorem 5.2]{bk:ad}, who previously showed the analogous Eisenstein series there are finite sums.
\end{abstract}


\section{Introduction}

\spoint
The theory of Eisenstein series on finite dimensional Lie groups has had profound applications in number theory, geometry, and mathematical physics.  Though Eisenstein series are initially defined as absolutely convergent sums in a right half plane or Weyl chamber, some of the most interesting applications occur at points of meromorphic continuation outside the range of convergence.
 The broad spectrum of their impact  includes:~the Langlands-Shahidi method of analytically continuing $L$-functions that occur in their Fourier expansions;
   calculations of volumes of fundamental domains through residues;
   constructing interesting representations (e.g., unitary) of noncompact Lie groups; estimates for moments of $L$-functions through multiple Dirichlet series; and applications to string theory.

Each of these areas has proposed applications from Eisenstein series on infinite-dimensional Kac-Moody groups.  The potential applications to $L$-functions are particularly striking, since the use of finite-dimensional groups has hit natural barriers both in the Langlands-Shahidi  and multiple Dirichlet series methods -- barriers which are tempting to circumvent through Kac-Moody groups.  More  detailed comments about these programs, along with a discussion of how our results pertain to them, are given in section~\ref{sec:introEminus} below.

This paper concerns the analytic continuation of cuspidal Eisenstein series on loop groups, for which absolute convergence in a natural right half plane has been established in papers of the first-named author.    A fundamental obstacle in this subject is that we do not yet know how to adapt the spectral theory methods  that obtain the crucial meromorphic continuation in the finite-dimensional case.
Our main theorem below gives an {\em entire} continuation of these series to the full complex plane by demonstrating their defining sums are in fact absolutely convergent everywhere.

\spoint \label{sec:intro-statement}
This paper builds on a  series of papers \cite{ga:abs,ga:ragh,ga:ms1, ga:ms2, ga:ms3, ga:ms4,ga:zuck} by the first-named author, which studied minimal parabolic Eisenstein series on loop groups. These series were shown to converge in a region
 analogous to the shifted Weyl chamber established by Godement \cite[\S3]{go:lang} in the finite-dimensional setting.  Moreover, by developing an analog of the Maass-Selberg relations, they were subsequently meromorphically continued to a larger region in   \cite{ga:ms1, ga:ms2, ga:ms3, ga:ms4}.  In contrast to the finite dimensional case, however,   there is a natural boundary or ``wall" of singularities at the boundary of this region, and thus these minimal parabolic loop Eisenstein series cannot be meromorphically continued any further.

On the other hand there  are also \emph{cuspidal loop Eisenstein series}, which are induced from cusp forms on the  Levi components of  parabolic subgroups of loop groups, e.g., maximal parabolics.  Their convergence in the Godement-type region was proved in \cite{ga:zuck} for  maximal parabolics, and  it was furthermore shown they can be analytically continued   up to the   aforementioned ``wall''   using an analog of the  Maass-Selberg relations.  Of note here is that   these relations for maximal parabolics consist of just a single   term, and consequently   cuspidal loop  Eisenstein series for maximal parabolic subgroups actually admit a  holomorphic  continuation up to the wall \cite{ga:zuck}.

In this paper we show that a much stronger result holds for cuspidal loop Eisenstein series  induced from the maximal parabolic subgroup whose Levi component  comes from the underlying finite-dimensional root system. The following result analytically continues these series  beyond the ``wall'':

\begin{nthm} \label{maininintro} Let $\phi$ be a spherical cusp form for a Chevalley group $G$ over $\Q$.  Then the cuspidally-induced Eisenstein series  $ E_{\vphi, \nu}$   (\ref{cusp:es}) on the loop group of $G$ converges absolutely for any $\nu \in \C$ to an entire function of $\nu$.  \end{nthm}

 \noindent   For example, this theorem applies to Eisenstein series on the loop group $E_9$ induced from spherical cusp forms $\phi$ on $E_8$.  Following the work \cite{ga:zuck} (which was written in the number field setting), a function-field analog of the above entirety result was announced by A.~Braverman and D.~Kazhdan  in \cite[Theorem 5.2(3)]{bk:ad}.
 As in many problems in automorphic forms, the underlying analytic mechanism responsible for the convergence differs greatly between the number field and function field setting, and indeed
  their proof relies on a geometric interpretation of the Eisenstein series unavailable over number fields.  In fact, at each fixed element in the appropriate symmetric space, their cuspidal Eisenstein series is  a finite sum (independent of the spectral parameter).  This finiteness crucially uses the observation of G.~Harder in the function field setting that cusp forms vanish outside of a compact subset of the fundamental domain.

  However, this vanishing outside of a compact set is far from true in the number field case, where cusp forms instead merely decay rapidly in the cusps -- a crucial aspect of our proof.   While rapid (e.g., polynomial) decay statements have well-known applications in the classical theory of automorphic forms and analytic number theory, such results turn out to be   insufficient for our purposes (see Remark~\ref{remark731}  as well as the end of this subsection for further details).  Instead we require an {\em exponential decay} result like Theorem~\ref{mainthm} (or its equivalent statement Theorem~\ref{decay:body})    to obtain   convergence in the number field setting.

Another key ingredient in our work is Theorem~\ref{iwineq} and its Corollary~\ref{cor:iwineq:2}, a growth estimate on the diagonal Iwasawa components in the classical and central directions as one varies over the loop analog of the discrete group.    In the function field setting, this estimate -- coupled with Harder's  compact  support theorem  -- is enough to deduce the Braverman-Kazhdan   result.  In the number field setting, the analysis is more involved and the entirety result we prove here is instead  shown by establishing a reduction to  the half-plane convergence result \cite{ga:zuck} recalled in Theorem~\ref{iw:iota}.

The statement of   Theorem~\ref{maininintro}   assumes that the cusp form $\phi$ is {\em spherical}, by which we mean that it is $K$-fixed at all places of $\Q$.  This restriction is undoubtedly not essential to our proof,
  but is made  because   of an obstacle coming from the local  theory of loop groups:~the usual definition of $K$-finite Eisenstein series (e.g., \cite[(6.3.1)]{shahidi}) relies on a matrix coefficient construction whose generalization to loop groups has not been developed.  Thus the  loop Eisenstein series induced from $K$-finite cusp forms have not even been defined, and so it is premature to study their convergence.
The $K$-finite condition at nonarchimedean places amounts to a congruence group condition, and hence a smaller group of summation in the Eisenstein series definition.  In the finite-dimensional setting, the absolute convergence of   congruence group Eisenstein series is an immediate consequence of the absolute convergence of  full-level Eisenstein series (simply by dropping terms).  Thus   relaxing the $K$-fixed condition at nonarchimedean places  may actually make absolute convergence  easier.

    The $K$-fixed condition at the
     archimedean place also arises in a paper by Kr\"otz and Opdam  \cite{Kr:opdam} on Bernstein's theorem that cusp forms decay exponentially.
     Though their methods of holomorphic  continuation of representations apply to arbitrary   archimedean $K$-finite cusp forms as well (e.g., in unpublished notes by those authors), there is currently no proof in the literature.   These decay results imply  our  important ingredient Theorem~\ref{decay:body}.    For completeness, and in order to allow our convergence argument to apply to $K$-finite Eisenstein series once their definition is eventually given, we have included an elementary proof of Theorem~\ref{decay:body}  for arbitrary   archimedean   $K$-finite cusp forms in the Appendix, which interjects only a small amount of hard analysis into the classical  methods for proving rapid decay.  Together, this provides  the full analytic machinery anticipated to be necessary to handle the convergence of the presently-undefined $K$-finite Eisenstein series.  The appendix also discusses some related questions about decay estimates posed by string theorists.

\spoint \label{sec:introEminus}  We conclude the introduction with some comments about possible extensions of the Langlands-Shahidi method using Theorem~\ref{sec:intro-statement}, in particular updating the status reports from \cite{ga:zuck} and \cite{bk:ad}.
 Recall that in his monograph {\em Euler Products} \cite{eulerproducts}, Langlands was led to introduce the notion of automorphic $L$-function in the course of studying the constant terms of  Eisenstein series.
   Namely, he found that the constant terms can be expressed in terms of ratios of a  special subfamily of automorphic $L$-functions $L(s,\phi,\rho)$:~those for which  $\phi$ is a cusp form on the Levi component $M$ of a maximal parabolic subgroup of $G$, and
 $\rho$ is a representation occurring in the decomposition of the  adjoint action of the Langlands dual group ${}^{L}M$ on the (dual of) the unipotent radical of the parabolic.  This generality is both a strength and a weakness, because it inherently limits which $L$-functions can arise (since all such $\rho$ are easily classified).  On the other hand, it is not difficult to show that {\em any} finite dimensional representation of a Chevalley group arises in the adjoint action of some ${}^{L}M$ in the Kac-Moody setting.  Already for loop groups several new and interesting $\rho$ arise (such as the 248-dimensional adjoint representation of $^LM=E_8$ inside $G=E_9$).

Exploiting the fact that the Eisenstein series on $G$ can be meromorphically continued  using spectral techniques (which do not apply directly to $L$-functions themselves), Langlands  deduced the meromorphy of the constant term and hence of its constituent $L$-functions.  This transference of meromorphy depends on the simple fact that the constant terms are obtained by integration over a compact space.  At the outset we should state that the biggest analytic issue to applying Theorem~\ref{sec:intro-statement} {\em ala} Langlands is the lack of a corresponding measure theory in the loop group setting,   due to noncompactness issues.

Affine Weyl groups are infinite and do not have longest elements.  This is reflected in the fact that loop groups have non-conjugate ``upper'' and ``lower'' triangular parabolics.  Thus each Levi subgroup $M$ arises two  very different ways:~either as the Levi of the ``upper triangular''  parabolic $P$ of Taylor loops with positive powers, or as the Levi of the opposite, ``lower triangular'' parabolic $P^-$ of Taylor loops with negative powers, where $P$ and $P^-$ are not conjugate.  Thus there are potentially two types of constant terms:~one involving integration over the unipotent radical of $P$, and one involving integration over the unipotent radical of $P^-$.    Both of these unipotent radicals are infinite-dimensional, and neither is locally compact. However, whereas the unipotent radical of $P$ can be equipped with a suitable projective limit structure  \cite[\S3]{ga:ragh} which makes its arithmetic quotient compact, the unipotent radical for $P^-$ is not known to have any such structure.

F.~Shahidi \cite{shahidi-infinite} showed that a naive extension of the Langlands-Shahidi method (involving the constant term for $P$) could not capture any nontrivial automorphic $L$-functions.
 A.~Braverman and D.~Kazhdan proposed that the constant term for $P^{-}$ contains the anticipated $L$-functions $L(s,\phi,\rho)$, where $\phi$ is a cusp form on $M$ and  $\rho$ occurs in the adjoint action of $^{L}M$.  The non-archimedean analogue of this computation (i.e., Gindikin-Karpelevich formula) has now been achieved in \cites{bfk, bgkp}.  However, globally there is a problem defining the measure on the unipotent radical of $P^-$  because of complications at the archimedean primes.

Braverman and  Kazhdan  further  proposed the following framework into which our entirety result naturally fits.    Corresponding to the two maximal  parabolics $P$ and $P^{-}$  one can consider two different Eisenstein series:~the (positive) cuspidal loop Eisenstein series  $E_{\phi, \nu}$ (\ref{cusp:es}) defined using $P,$ and the negative cuspidal loop Eisenstein series $E^-_{\phi, \nu}$ defined using  $P^-.$
At present $E^{-}_{\nu,\phi}$ is only defined in the function field case, and only for $\Re\!(\nu)$ large at that (see \cite{bk:ad} for the original geometric definition or \cite{gmp} for a purely group-theoretical definition).   Note that naively switching $P$ to $P^-$ in (\ref{cusp:es}) causes the series to diverge and so $E^{-}_{\nu,\phi}$ must be defined as a ``regularization'' of this divergent series.  The divergence is akin to the one which would result by choosing the degree parameter $r$ from (\ref{eta(s)}) to be \emph{negative} in the definition (\ref{cusp:es}) of $E_{\nu,\phi}$.
Theorem~\ref{maininintro} shows that  $E_{\phi, \nu}$ is   entire.  However, the series $E^-_{\phi, \nu}$
is predicted to    have an analytic continuation only up to a   ``wall" similar to the one discussed at the beginning of \S\ref{sec:intro-statement}.  Furthermore, there is an expected functional equation relating the positive and negative Eisenstein series  of the form
\be{eplus-emin} E_{\nu, \phi} \ \  = \ \  \Xi(\nu, \phi) E^-_{\nu, \phi}\,, \ee
 where $\Xi(\nu, \phi)$ can be written in terms of  automorphic $L$-functions.  The precise form of $\Xi(\nu, \phi)$ is suggested by the Gindikin-Karpelevich integral from \cite{bgkp, bfk}.

 One advantage of switching from $E_{\nu,\phi}$ to $E^-_{\nu, \phi}$ is that the constant term expected to give  new $L$-functions  switches  from $P^{-}$ to $P$.  As mentioned above, the latter now has
sufficiently adequate measure theory to transfer holomorphy.  Thus holomorphic continuation of $E^-_{\nu, \phi}$ would give new analytic continuations of $L$-functions in its constant terms.  Alternatively, if the functional equation (\ref{eplus-emin}) can be proven in some domain, then the automorphic $L$-functions in $\Xi(\nu, \phi)$ would be continuable to $E^-_{\nu, \phi}$'s domain of holomorphy (because of the holomorphy of $E_{\nu,\phi}$ proved in Theorem~\ref{maininintro}).
In the function field setting, currently one can define $E^-_{\nu, \phi}$ \cite{bk:ad}, take its constant terms over $P$ \cite{bgkp,bfk}, and  prove the functional equation (\ref{eplus-emin}).  However, the analytic continuation of $E^-_{\nu, \phi}$ is presently unknown, and    {\em none} of these ingredients is known   in the number field setting.

In summary, there are two approaches to obtaining new meromorphic continuations of $L$-functions using loop Eisenstein series:~either developing an integration theory over the unipotent radical of $P^{-}$ to transfer the holomorphy of $E_{\nu,\phi}$, or instead defining and developing the analytic properties of $E^{-}_{\nu,\phi}$.  We note that the (presently missing) integration theory could also yield meromorphic continuations of non-constant   Fourier-Whittaker   coefficients.  If applied to Eisenstein series on covers of loop groups, this could potentially be a source of new meromorphic continuations of multiple Dirichlet series,    such as ones used for studying moments of Dirichlet $L$-functions (see \cite{Bucur-Diaconu}, which has a detailed discussion of the fourth moment problem for quadratic characters.).

\vspace{0.15in}

\emph{Acknowledgements:}  We would like to thank Alexander Braverman for informing us of his results with David Kazhdan and of Bernstein's exponential decay estimates. We would also like to thank William Casselman, Solomon Friedburg, Michael B.~Green, Jeffrey Hoffstein, Bernhard Kr\"otz,  Peter Sarnak, Wilfried Schmid, and Freydoon Shahidi for their helpful comments.

\section{Notation} \label{section:notation}

In this section we shall recall some background material which is used later in the paper.
Part 2A concerns the finite-dimensional situation, part 2B discusses affine loop groups, part 2C focuses on the affine Weyl group, and part 2D summarizes some useful decompositions.  Finally, part 2E redescribes some of this material from the adelic point of view.  A general reference for most of this material  is Kac's book~\cite{kac}.

\vspace{.3cm}
\begin{center}
{\bf A. Finite Dimensional Lie Algebras and Groups}
\end{center}

\spoint \label{finrtsystems} Let $\bf{C}$ be an irreducible $\ell\times \ell$ classical Cartan matrix. Let $\mf{g}$ be the corresponding real, split, simple, finite-dimensional  Lie algebra of rank $\ell.$ Choose a Cartan subalgebra $\mf{h} \subset \mf{g}$ and denote the set of roots with respect to $\mf{h}$ by $\Delta.$  Let $\Pi= \{ \alpha_1, \ldots, \alpha_{\ell} \}\subset {\frak h}^*$ denote a fixed choice of   positive simple roots, and $\Pi^{\vee} = \{ h_1, \ldots, h_{\ell} \}\subset {\frak h}$ the corresponding set of simple coroots. We let $\Delta_{+}$ and $\Delta_{-}$ be the sets of positive and negative roots, respectively. For a root $\alpha \in \Delta$ we denote by $h_{\alpha}$ the corresponding coroot. Let $(\cdot, \cdot)$ denote the Killing form on $\mf{h}$ normalized so that $(\alpha_0, \alpha_0)=2$, where $\alpha_0$ is the highest root of $\mf{g}$, and write $\langle \cdot,\cdot \rangle$ for the natural pairing between ${\frak h}^*$ and $\frak h$.

Let $W$ be the corresponding Weyl group of the root system defined above  and $Q$ and $Q^{\vee}$ the root and coroot lattices, which are spanned over $\zee$ by $\Pi$ and $\Pi^{\vee}$, respectively.  We define  $\Lambda \supset Q$ and $\Lambda^{\vee}\supset Q^\vee$ to be the weight and coweight lattices, respectively. The fundamental weights will be denoted by $\omega_i$ for $i=1, \ldots, \ell$;  recall that these  are defined by the conditions that
\begin{equation}\label{new2.1}
 \langle \omega_i, h_j \rangle \ \  = \ \ \left\{
                                                          \begin{array}{ll}
                                                            1\,, & i\,=\,j\, \\
                                                            0\,, & i\,\neq\,j\,
                                                          \end{array}
                                                        \right. \ \ \
 \text{ for ~} 1 \,\le \, i,j\, \le \, \ell \,.
\end{equation}
We also define the fundamental coweights $\omega_j^{\vee} \in \mf{h}$ for $j=1, \ldots, \ell$ by the conditions  that
\begin{equation}\label{new2.2}
 \langle \alpha_i, \omega^{\vee}_j \rangle \ \  = \ \ \left\{
                                                          \begin{array}{ll}
                                                            1\,, & i\,=\,j\, \\
                                                            0\,, & i\,\neq\,j\,
                                                          \end{array}
                                                        \right. \ \ \  \text{ for ~} 1 \,\le \, i,j\, \le \, \ell \,.
\end{equation}
Note that we have \be{cowt:corts} ( \omega_i^{\vee}, h_j ) =\ \ \left\{
                                                          \begin{array}{ll}
                                                            \f{2}{(\alpha_i,\alpha_i)}\,, & i\,=\,j\, \\
                                                            \ \ \ \ 0\,, & i\,\neq\,j\,
                                                          \end{array}
                                                        \right. \ \ \  \text{ for ~} 1 \,\le \, i,j\, \le \, \ell \,. \ee
As usual we set  \be{rho} \rho \ \ = \  \ \smallf{1}{2} \sum_{\alpha  \, \in \,  \Delta_+} \alpha
 \ \ = \ \ \sum_{j\,=\,1}^\ell\omega_j
\,, \ee
 which satisfies the condition \be{rho:h_i} \la \rho, h_i \ra  \ \ = \ \  1\; \text{ for ~} 1\,\le \, i \, \le \,   \ell \,.\ee  We record here the following elementary statement which will be useful later. It is equivalent to the positivity of the inverse of the Cartan matrix.

\begin{nlem}[\cite{ragh}*{p.111, Lemma 6}] \label{ragh} Each fundamental weight $\omega\in\{\omega_1,\ldots,\omega_\ell\}$ can be written as a strictly positive linear combination of positive simple roots $\alpha_1,\ldots,\alpha_\ell$.
\end{nlem}

\spoint \label{sd1} Let $G$ be a finite dimensional, simple, real Chevalley group with Lie algebra $\frak g$ constructed as in \cite{stein:yale}. Let $B \subset G$ be a Borel subgroup with unipotent radical $U$ and split torus $T.$   Let $A \subset T$ be the connected component of the identity of $T$, and assume (as we may) that $B$ is arranged so that $\Lie(A)=\frak h$ and that $\Lie(U)$ is spanned by the root vectors for $\D_+$.  Then $G$ has a maximal compact subgroup $K$ with Lie algebra orthogonal to $\frak h$ such that the   Iwasawa decomposition
\be{iw:fd} G \ \ = \ \  U \,  A\,  K \ee
holds with uniqueness of expression.
 We denote the natural projection onto the $A$-factor by the map $g\mapsto \iw_A(g)$.
Any linear functional $\lambda: \mf{h} \rr \C$ gives rise to a
quasi-character   $a \mapsto a^{\lambda}$ of $A$ via
 \be{atolambdef} a^{\lambda} \ \ := \ \ \exp( \langle  \lambda , \ln a \rangle ). \ee

\spoint \label{sd2} By construction, $G \subset \Aut(V)$ for some highest weight module $V.$
Moreover, for the Chevalley lattice $V_{\zee} \subset V$ constructed as in \cite{stein:yale}, let $\Gamma \subset G$ denote  the stabilizer of $V_{\zee}.$ We then use the following notion of Siegel set as defined in \cite{bor:groupes}. Let $U_{\mc{D}} \subset U$ denote  a fundamental domain for the action of $\Gamma \cap U$ on $U.$ For any $t > 0$ we set \be{Atdef} A_t \ \  :=  \ \ \{ a \in A \, |\, a^{\alpha_i} > t \,  \ \text{for each}\  1\,\le\,i \,\le\,\ell \}\,. \ee  By definition, a Siegel set has the form $\mf{S}_t =  U_\mc{D}  A_t  K$.
 The following is shown in \cite{bor:groupes}.

\begin{nprop} Suppose that $t < \sqrt{3}/2.$ Then for every $g \in G$ there exists $\gamma \in \Gamma$ such that $\gamma g \in \mf{S}_t$, i.e., the $\G$-translates of ${\frak S}_t$ cover $G$. \end{nprop}

\tpoint{Remark}\label{remark231}   One of the key points in the proof of this proposition is the following minimum principle (which will be important for us later).  Let $V^\rho$ be the highest weight representation of $G$ corresponding to the dominant integral weight $\rho: \mf{h} \rr \C$ from (\ref{rho}).
 This representation comes equipped with  a $K$-invariant norm $|| \cdot ||$. Let $v_{\rho} \in V^{\rho}$ denote a highest weight vector. Then for any $g \in G,$ it is not hard to see  that the function $\Psi_g :  \gamma  \mapsto \| g^{-1}\gamma   v_{\rho} \|$ achieves a minimum value on $\Gamma$ (see \cite{bor:groupes}). Furthermore, this minimum is achieved by some
  $\gamma \in \Gamma$  such that $\gamma^{-1} g \in \mf{S}_t$.


\vspace{.5cm}
\begin{center}
{\bf B. Affine Lie Algebras }
\end{center}

\spoint Let $\wh{\bf{C}}$ denote the $(\ell+1)\times (\ell+1)$ affine Cartan matrix corresponding to $\bf C,$ and $\hf{g}$
the one-dimensional central extension of the loop algebra corresponding to the finite dimensional Lie algebra $\frak{g}$
associated to {\bf C}.
   The extension  \begin{equation}\label{new2.4.1}
\hf{g}^e  \ \ = \ \  \hf{g} \,\oplus \, \R \dd
\end{equation}
of $\hat{\frak{g}}$ by the \emph{degree derivation} $\dd$ is the (untwisted) affine Kac-Moody algebra associated to $\car$ (see [Kac, Chapters 6 and 7]).

 Let $\hf{h} \subset \hf{g}$ be a Cartan subalgebra  and define  the extended Cartan subalgebra \be{h:ext} \hf{h}^e  \ \ := \ \  \hf{h} \, \oplus \, \R \dd \,. \ee
 Let $\hed$ denote the (real) algebraic dual of $\he$ and continue to denote the natural pairing  as \be{dualpair} \la \cdot, \cdot \ra \, :\, \hed \times \he  \ \rr  \ \R\,. \ee
 The set of simple affine roots will be denoted by  \be{simp:aff:roots} \wh{\Pi} \ \ = \ \  \{ a_1, \ldots, a_{\ell+1} \}  \ \ \subset \ \  \hed\,. \ee Similarly, we write \be{simp:aff:coroots} \wh{\Pi}^{\vee}  \ \ = \ \  \{ h_1, \ldots, h_{\ell+1} \}  \ \ \subset \ \  \hf{h} \ \ \subset \ \ \he \ee for the set of simple affine coroots.  Note that the simple roots $a_i: \hf{h}^e \rr \R$ satisfy the relations\begin{equation}\label{new2.16}
\aligned \la a_i , \dd \ra & \ \ = \ \   0  \ \ \text{ for } \ i \, =\, 1,\, \ldots, \,\ell \\ \text{and} \ \ \  \ \  \
\la a_{\ell+1} , \dd \ra &\ \ = \ \   1\,.
\endaligned
\end{equation}
A more explicit description of the $a_i$ will be given in section~\ref{sec:2.5}.

We define the affine root lattice  as  $\hQ = \zee a_1   +   \cdots   +   \zee a_{\ell+1}$
and the affine coroot lattice as  $\hQ^{\vee} =  \zee h_1 + \cdots  + \zee h_{\ell+1}$.
We shall denote the subset of non-negative integral linear combinations of the $a_i$ (respectively, $h_i$) as $\hQ_+$ (respectively,  $\hQ^{\vee}_+$).  The integral weight lattice is defined as
\begin{equation}\label{integralweightlatticedef}
\hLambda  \ \  :=  \ \  \{ \lambda \in \hf{h}^* \ | \   \langle \lambda, h_i \rangle \,  \in \,  \zee  \, \ \text{ for } \, i\,=\, 1, \ldots, \ell+1  \}\,.
\end{equation} We  regard $\hLambda$  as a subset of $\hed$  by declaring that \be{2.20} \la \lambda , \dd \ra \ \  = \ \  0 \,   \ \ \ \text{for} \ \  \lambda \,\in\,\wh{\Lambda}\,. \ee  The lattice $\hLambda$ is spanned by the fundamental affine weights  $\Lambda_1, \ldots, \Lambda_{\ell+1},$ which are defined by the conditions that  \be{aff:wts} \langle \Lambda_i, h_j \rangle \ \  = \ \ \left\{
                                                          \begin{array}{ll}
                                                            1\,, & i\,=\,j\, \\
                                                            0\,, & i\,\neq\,j\,
                                                          \end{array}
                                                        \right. \ \ \
 \text{ for ~} 1 \,\le \, i,j\, \le \, \ell+1 \,.\ee
Note that the dual space $(\hf{h}^e)^*$ is spanned by $a_1, a_2, \cdots, a_{\ell+1}, \Lambda_{\ell+1}.$

For any  $a\in \wh{Q}$ define
 \begin{equation}\label{rootspacedef}
 \hf{g}^{a} \ \ := \ \  \{ x \in \hf{g} \ |\  [h, x] \,= \,\langle a, h \rangle x\, , \  \text{ for all } h \in \hf{h}^e \}\,.
 \end{equation}
  The set of nonzero $a\in \hQ$ such that $\dim  \hf{g}^{a} > 0$ will be called the roots of $\hf{g}$ and denoted by $\wh{\Delta}.$
   The sets $\wh{\Delta}_+:= \wh{\Delta} \cap \hQ_+$ and $\hdelta_{-}:= \hdelta \cap ( - \hQ_{+})$ will be called the sets of positive and negative  affine roots, respectively.  We have that  $\wh{\Delta}=\wh{\Delta}_+\sqcup \wh{\Delta}_-$, i.e.,  every element in $\hdelta$ can be written as a linear combination of elements from $\hPi$ with all positive integral or all negative integral coefficients.

For each $i \in \{ 1, 2, \ldots, \ell+1 \}$ we define the reflection $w_i: \hed \rr \hed$ by the formula \be{sim:ref}  w_i : \lambda \   \mapsto  \  \lambda  \, - \, \la \lambda, h_i \ra \,a_i \ \ \ \ \text{ for } \ \  \  \lambda \, \in\, \hed\,. \ee The Weyl group $\aw \subset \Aut(\hed)$  is the group generated by the elements $w_1,\ldots,w_{\ell+1}$.  The dual action of $\aw$ on $\he$ is defined by the formula
 \begin{equation}\label{affw:hact} \  \la \lambda, w \cdot h \ra \ \  =  \ \   \la w^{-1} \lambda, h \ra  \ \ \ \ \text{ for all }  \ \lambda \,\in\, \hed \, , \  h \,\in \,\he \, , \ \text{and} \  w \,\in \,\aw\,. \end{equation}
The roots decompose as
 \begin{equation}\label{DDpDm}
  \wh{\Delta} \ \  = \ \  \  \widehat{\Delta}_W  \  \sqcup \ \wh{\Delta}_I  \,, \\
 \end{equation}
where $\widehat{\Delta}_W$ (known as the ``real roots'' or ``Weyl roots'') are the $\wh{W}$-translates of  $\wh{\Pi}$, and $\wh{\Delta}_I$ (known as the ``imaginary roots'') is its complement in $\wh{\Delta}$.  These sets will be described explicitly in (\ref{deltawhat})-(\ref{deltaihat}) below.    Each imaginary root is fixed by $\wh{W}$.  The space $\hf{g}^{a}$ is 1-dimensional for $a\in \widehat{\Delta}_W$ and is $\ell$-dimensional for $a\in\widehat{\Delta}_I$.
Coroots of elements
 $a \in \hdelta_W$ can be defined by the formula
\begin{equation}\label{gencorootdef}
 h_a \ \ := \ \  w^{-1} h_j \,  \in \, \hf{h}\, ,
\end{equation}
where $w$ is an element of $\aw$  such that $w a =a_j$ (this is shown to be well-defined in \cite[\S5.1]{kac}).

\newcommand{\cc}{\mathbf{c}}

\spoint\label{sec:2.5} We shall now  give a more concrete description of the affine roots and coroots in terms of the underlying finite dimensional Lie algebra $\mf{g}$. It is known that $\hf{g}$ is a one-dimensional central extension of the \emph{loop algebra} of the finite dimensional Lie algebra $\mf{g}.$ We denote this one-dimensional center by $\R \cc,$ where $\cc \in \hf{h}.$
The Cartan subalgebra in the extended affine algebra $\hf{g}^e$ can then be written as \be{extended:cartan:c} \hf{h}^e  \ \ = \ \  \R \cc  \, \oplus \,  \mf{h} \, \oplus \, \R \dd\,, \ee
 in which we have assumed (as we may) that $\frak h$ coincides with the fixed Cartan subalgebra of $\frak g$ chosen in section~\ref{finrtsystems}. The finite dimensional roots $\alpha \in \Delta$ can then be extended to elements of $\hed$ by stipulating that
  \be{2.28} \la \alpha, \cc \ra \ = \  \la \alpha, \dd \ra \ = \ 0 \ \  \     \text{ for each  }  \ \alpha \, \in \,  \Delta \, . \ee
    In particular the element $\rho$ defined in (\ref{rho}), which for us always represents an object from the classical group, extends to an element of $\hed$ that is trivial on $\cc$ and $\dd$.
  We may then identify the first $\ell$ simple roots of the affine Lie algebra with those of its classical counterpart:
   \be{a:alpha}a_i  \ \ = \ \  \alpha_i \ \text{ for } \ i =1, 2, \ldots, \ell\,. \ee We likewise extend the fundamental weights $\omega_j$ for $j =1, \ldots, \ell$   to elements of  $\hed$ by setting \be{omega:ext} \la \omega_j , \cc \ra  \ = \  \la \omega_j, \dd \ra  \ = \  0\,. \ee Also, we can identity the affine coroots $h_1, h_2, \ldots, h_{\ell}$ with the corresponding classical coroots defined in \S\ref{finrtsystems} under the same name. It remains to describe both $a_{\ell+1}$ and $h_{\ell+1}$ in terms of the underlying finite-dimensional root data.

Let  $\iota \in \wh{\Delta}$ be the minimal positive imaginary  root. It is characterized as the unique linear map $\iota \in \hed$ satisfying the condition that
\begin{equation}\label{iotaX}
 \aligned \la \iota, X \ra & \ \ = \ \  0 \ \ \ \  \text{ for } \  X \, \in \,\R \cc \oplus \mf{h}  \\
  \text{and} \ \ \ \ \ \ \la \iota , \dd \ra & \ \ = \ \ 1\,. \endaligned
\end{equation}
 We then have the following explicit description of the real and imaginary roots of $\hf{g}^e$:
   \be{affine:roots}
  \widehat{\Delta}_W & = & \ \  \{ \alpha \, + \, n \,\iota \,\mid \, \alpha \, \in \,  \Delta\,, \ n \, \in \,  \Z \} \label{deltawhat} \\
  \wh{\Delta}_I  \ & = & \ \  \{n\,\iota \,|\, n\,\neq\,0  \,,\, n \in \, \Z \}\,. \label{deltaihat}
\ee
 The classical roots $\Delta$ can be regarded as the subset of $\wh{\Delta}_W$ with  $n=0$ in this parametrization. One then has \be{newroot} a_{\ell+1} \ \ = \ \  -\, \alpha_0  \, + \,  \iota\,, \ee
where we recall that $\alpha_0$ denotes the highest root of the underlying finite-dimensional root system.
%

We shall now give a similar description of $h_{\ell+1}.$ To do so, we recall from \cite[(6.2.1)]{kac} that one can define a symmetric,  non-degenerate, invariant bilinear form $(\cdot | \cdot)$ on $\hf{h}^e.$  The form is first defined on $\hf{h}$ in terms of certain labels of the affine Dynkin diagram (coming from the coefficients of $\a_0$ when expanded as a sum of the finite simple roots), and then is extended to all of $\hf{h}^e$ by setting
  \begin{equation}\label{2.35}
  \aligned
  (h_i \mid \dd) & \ \ = \ \  0\,, \ \ \   \text{ for } \ i\,=\, \, 1, \ldots, \ell\,, \\ (h_{\ell+1} \mid \dd ) & \ \ = \ \ 1\,, \\
\text{and} \ \ \ \ \ \ \ \ \ \ \ \ \ \ \ \ \
(\dd \mid \dd) &\ \ = \ \  0\,. \endaligned
  \end{equation}
 The form $(\cdot \mid \cdot)$ also induces a symmetric bilinear form on $\hed$, which we continue to  denote by $(\cdot \mid \cdot)$, and which  has the following characterization:
\begin{equation}\label{norm:affkilling}
 \aligned (a_i \mid a_j) & \ \ = \ \ (\alpha_i, \alpha_j)\,,  \ \  \  \text{ for } \ i, j \, \in \, \{ 1, \ldots, \ell \}\,, \\
	(a_i \mid a_{\ell + 1}) & \ \ = \ \  (\alpha_i, - \alpha_0)\,,  \ \ \text{ for } \  i\,\in\,\{1, \ldots, \ell\} \, ,   \\
\text{and} \ \ \ \ \ \ \	(a_{\ell+1} \mid a_{\ell+1}) & \ \ = \ \  (\alpha_0, \alpha_0) \ \  = \ \  2 \,. \endaligned
\end{equation} It is easy to see that the above equations imply that \be{norm:affkilling:iota} (\iota \mid \iota)  \ = \  0 \ \text{  and  } \  (\iota \mid a_i ) \  = \ 0 \text{ for } i=1, \ldots, \ell. \ee
 Note that together with the relations
\begin{equation}\label{2.37}
 \aligned (\Lambda_{\ell+1} \mid \Lambda_{\ell+1} ) & \ \ = \ \  0\,, \\
	(a_i \mid \Lambda_{\ell+1} ) & \ \ = \ \  0 \,, \ \ \ \text{ for } i =1, \ldots, \ell\,, \\
\text{and} \ \ \ \ \ \ \   	(a_{\ell+1} \mid \Lambda_{\ell+1} ) &\ \ = \ \   1\,, \endaligned
\end{equation}
(\ref{norm:affkilling}) and (\ref{norm:affkilling:iota}) completely specify the form $(\cdot | \cdot)$ on $\hed$ (and hence $\hf{h}^e$).  We define normalized coroots by the formula
\be{hprimeh}
  h'_{a_i} \ \  := \ \  \smallf{ (a_i| a_i)}{2}\,h_i \, , \ \ i\,\in\,\{1,\,\ldots,\,\ell+1\}\,. \ee  More generally for
  \begin{equation}\label{2.39}
  b  \ \ = \ \  \sum_{i\,=\,1}^{\ell+1} \kappa_i \,a_i
\end{equation}
  we set
  \begin{equation}\label{2.40}
   h'_b \ \ := \ \  \sum_{i\,=\,1}^{\ell+1} \kappa_i \,h'_{a_i}\,.
  \end{equation}  Using (\ref{a:alpha}) and (\ref{newroot}) these conventions  give a definition for   $h'_{\iota}$.

  For any real root $b$ we have defined the corresponding coroot $h_b$ in (\ref{gencorootdef}) in terms of the Weyl group action. One then has  that \be{h:h'} h'_b  \ \ = \ \ \smallf{ (b|b) }{2}\,h_b \ee  (see for example \cite[(5.1.1)]{kac}).
   We now set
\begin{equation}\label{2.41}
 h_{\iota} \ \  := \ \  \smallf{2}{(\alpha_0, \alpha_0)}\, h'_{\iota} \ \  = \ \  h'_{\iota}\,,
\end{equation}
 using $(\alpha_0,\alpha_0)=2$. Suppose $a = \alpha + n \iota \in \wh{\Delta}_W$ with $\alpha \in \Delta$ and $n \in \zee$ as in (\ref{deltawhat}).  Then in fact \begin{equation} \label{h_a:h'} h_a  \ \ = \ \  \smallf{(\alpha \mid \alpha)}{(a \mid a)} \,h_{\alpha}  \ + \  \smallf{2}{ (a \mid a) } \, n \,h_{\iota}\, \ \ = \ \  h_{\alpha}  \ + \  \smallf{2}{ (\alpha \mid \alpha) } \, n \,h_{\iota}\,,
\end{equation} where in the second equality we have used the fact that $(a \mid a) = ( \alpha \mid \alpha)$, a consequence of (\ref{norm:affkilling:iota}).  In particular we also have the formula\be{newcoroot} h_{\ell+1}  \ \ =  \ \ h_{- \alpha_0}  \ +  \ h_{\iota}\,, \ee
  in analogy to (\ref{newroot}).
 One may check that $h_{\iota}$ is also a generator for the one-dimensional center $\R {\bf c}$ of $\hf{g}^e$ and that we have the direct
  sum decompositions
  \be{extended:cartan} \hf{h}^e  \ \ = \ \   \R h_\iota \, \oplus \, \mf{h} \, \oplus \,  \R \dd\ \ \ \ \text{and} \  \ \ \
 \hf{h} \ \ = \ \ \R h_\iota\,\oplus\,\mf{h}
  \ee similar to (\ref{extended:cartan:c}).
This   decomposition of $\hf{h}^e$ will be frequently used later in this paper.

\vspace{.5cm}
\begin{center} {\bf C. Lemmas on the Affine Weyl Group}\end{center}

\spoint\label{sec:What} The affine Weyl group $\aw$ was defined  above  as the group generated by the reflections $w_1,\ldots,w_{\ell+1}$ from (\ref{sim:ref}). It also has a more classical description as the semidirect product \be{affweyl:semidirect} \wh{W}  \ \ = \ \  W \ltimes Q^{\vee}. \ee
More concretely, the elements $b$ of the coroot lattice $Q^{\vee}$  correspond to translations $T_b$ in the Weyl group $\wh{W}.$ Recall that $T_b$ fixes $\iota$, as do all elements in $\aw$. The action of $T_b$ on $\l\in \operatorname{Span}(a_1,\ldots,a_\ell)$ is given by the formula
 \begin{equation}\label{Tbdef}
 \qquad\qquad\qquad\qquad\qquad
    T_b\,:\,\l  \ \mapsto \ \l\,+\,\la \l, b \ra \iota \qquad\qquad\qquad\qquad\qquad
    \text{(\cite[(4.2)]{ga:ragh}).}
 \end{equation} It is also possible to obtain a general formula for the action of $T_b$ on an element in $\hed,$ though we will not need it here.

  We shall now state
 a formula for the action of $\wh{W}$ on  $\he$, which we shall make frequent use of later.
   The general element of  $\hf{h}^e$ can be written in terms of the decomposition (\ref{extended:cartan}) as $m h_{\iota}  + h +  r \dd$.  If we  use (\ref{affweyl:semidirect}) to factor  $w \in \wh{W}$ as $w= \t{w} T_b$ for some  $\t{w} \in W$ and $b \in Q^{\vee},$ then we have
    \begin{multline}\label{affw:actionandcenter} w \cdot \left( m h_{\iota}  \, + \, h \,  + \,  r \,\dd\right )  \ \ = \ \  \left[-\,\frac{r\,(b, b)}{2}  \, + \,  (h, b) \, + \,  m\right] h_{\iota}  \  - \   r \, \t{w}(b) + \t{w}(h)  \  + \   r\, \dd \end{multline}
 (see \cite[p. 309]{ga:ragh}).  Note that in particular $\wh{W}$ fixes $h_\iota$, $\iota$, and $\cc$.  Combined with the fact that (\ref{gencorootdef}) is well-defined, this shows that the coroots satisfy
 \begin{equation}\label{weyloncoroot}
    w\,h_a \ \ = \ \ h_{wa}
 \end{equation}
 for any $w\in \wh{W}$ and $a\in\wh{\Delta}$.

\spoint For each $w \in \wh{W}$ we denote by $\ell(w)$ the length of $w$, i.e., the minimal length of a word in $w_1,\ldots,w_{\ell+1}$ which represents $w$.  We then have the following  estimate:

\begin{nlem} \label{l(w):ineq} Let $\|\cdot\|$ be an arbitrary norm on $\frak h$.  Then there exist constants $E,\, E' > 0$ depending only on the (finite-dimensional) root system $\Delta$ and the choice of norm $\|\cdot\|$ such that
\begin{equation}\label{new2.49}
    \ E'   \,|| b ||\, -\, \#(\Delta_+) \, \ \ \leq \ \ \ell(w) \ \ = \ \ \ell(w^{-1})  \ \ \leq  \
 \ E    \,|| b ||\, +\, \#(\Delta_+) \,,
\end{equation}
 for any element   $w \in \wh{W}$ of the form   $w=T_b \t{w}$ or $w=\t{w}T_b$, where  $b \in Q^{\vee}$ and $\t{w} \in W.$
 \end{nlem}
\begin{proof}
 Assume that $w$ has the form $T_b\t{w}$ (the statement for $w=\t{w}T_b$ is equivalent).
 From \cite[Proposition 1.23]{iwamat} one has \be{iwm:len}  \ell(w)  \ \ = \ \ \sum_{ \alpha \, \in \,  \Delta_+} | \la \alpha, b \ra - \chi_{\Delta_-}(\t{w}^{-1} \alpha) |\,, \ee where $\chi_{\Delta_-}$ is the characteristic function of $\Delta_-.$  By the triangle inequality
 \be{l(w):ineq1}  \ell(w)  \ \ \leq  \ \   \sum_{\alpha \, \in \,  \Delta_+} (| \la \alpha, b \ra | + 1) \ \ =  \ \
 \sum_{\alpha \, \in \,  \Delta_+} | \la \alpha, b \ra | \ + \ \#(\Delta_+)
  \,. \ee Similarly, we deduce  from  (\ref{iwm:len}) that
  \be{l(w):ineq1b}  \ell(w)  \ \ \geq  \ \   \sum_{\alpha \, \in \,  \Delta_+} (| \la \alpha, b \ra | - 1) \ \ = \ \
 \sum_{\alpha \, \in \,  \Delta_+} | \la \alpha, b \ra | \ - \ \#(\Delta_+)
  \,. \ee
Writing $b = \sum_{i=1}^\ell d_i \omega_i^{\vee}$ in terms of the fundamental coweights $\omega_i^{\vee}$  and letting \be{2.52} M \ \ = \ \ \max_{\srel{ \ \alpha\,\in\,\Delta_+}{1\le i \le \ell}}\la \alpha,\omega_i^{\vee} \ra\, ,\ee
we have 
 $| \la \alpha, b \ra |\leq    M  \,\sum_{i=1}^{\ell} | d_i |$
  for any  $\alpha \in \Delta_+$.  On the other hand,
   \be{2.54} \sum_{\alpha \, \in \,  \Delta_+} | \la \alpha, b \ra | \ \  \geq
   \ \  \sum_{\alpha \, \in \,  \Pi} | \la \alpha, b \ra | \ \ =
   \ \  \sum_{i\,=\,1}^\ell |   \la   \alpha_i, b \ra | \ \  =  \ \ \sum_{i\,=\,1}^{\ell} | d_i |\,. \ee
As all norms on a finite dimensional vector space are equivalent, the assertions of the lemma follow. \end{proof}

\spoint \label{flip}   For any element $w \in \wh{W}$ we set
 \begin{equation}\label{Deltawdef}
 \wh{\Delta}_w  \ \ =  \ \ \{ a \in \wh{\Delta}_+\,\mid\, w^{-1} a < 0 \}\,.
\end{equation}
It is a standard result that
\begin{equation}\label{lengthandcard}
\#( \wh{\Delta}_w ) \ \ = \ \     \ell(w) \ \ = \ \ \ell(w^{-1}) \,.
\end{equation}
Moreover, given a  reduced decomposition $w= w_{i_r} \ldots w_{i_1},$  $r=\ell(w)$, the set $\wh{\Delta}_w$ can be described as \be{delta_w} \wh{\Delta}_w  \ \ = \ \  \{ \beta_1, \ldots, \beta_r \}\,, \ee where 
$\beta_j =  w_{i_r} \cdots w_{i_{j+1}} (a_{i_j})$
for $j=1, \ldots, r.$ We also set \be{neg:flipped} \wh{\Delta}_{-, w} \ \  = \ \  \{ a \in \wh{\Delta}_-\, |\, w a > 0 \} \ \ = \ \  w^{-1}\,\wh{\Delta}_w  \,. \ee   Similarly to (\ref{delta_w}) it equals
 \be{delta_-w} \wh{\Delta}_{-, w}  \ \ = \ \  \{ \gamma_1, \ldots, \gamma_r \}, \ee where
$\gamma_j  = w^{-1} \beta_j =   w_{i_1} \cdots w_{i_{j}}(a_{i_j}) =
- w_{i_1} \cdots w_{i_{j-1}}(a_{i_j})$
for $j=1,\ldots,r$.

\spoint Considering the classical Weyl group $W$ as the subgroup of $ \wh{W}$ generated by $w_1,\ldots,w_\ell$, let $W^{\theta}$ denote the {\em Kostant coset representatives}    for   the quotient $W \backslash \wh{W} $:
 \be{kostant} W^\theta \ \ = \ \ \{  w \,\in\,\wh{W} \,\mid \,  w^{-1} \alpha_i > 0 \,  \ \text{ for } \,  \   i=1, \ldots,\ell \}\,. \ee
 For any $i =1, \ldots, \ell$ and $w \in W^{\theta}$ let us write
 \begin{multline}\label{kappa_i} w^{-1} \alpha_i  \ \ = \ \  \sigma_i \  + \ \kappa_i(w^{-1}) \iota\,, \ \ \ \text{ where } \  \sigma_i \,\in\, \Delta \ \ \, \text{and} \ \  \kappa_i(w^{-1}) \, \in\, \zee_{\geq 0}\,, \end{multline}
which is possible since the Weyl translates of the $\alpha_i$ lie in $\wh{\Delta}_W$ (see (\ref{deltawhat})).

\begin{nlem} \label{kap:ineq} For any $w \in W^{\theta}$ and $i=1, \ldots, \ell,$ we have  that
\be{map:ineq} \kappa_i(w^{-1})  \ \ \leq  \ \ \ell(w) \ + \  1 \ \ = \ \  \ell(w^{-1}) + 1  \,.\ee \end{nlem}

\begin{proof}
 Suppose first that $\sigma_i$ in (\ref{kappa_i}) is positive.  Then  \be{w:app} w^{-1} ( - \alpha_i + n \iota ) \ \  = \ \  - \sigma_i  \ + \  (n - \kappa_i(w^{-1}) )\, \iota \ee
since $\wh{W}$ preserves the imaginary root $\iota$.
   The root $-\alpha_i + n \iota$ is positive if $n >0$, while  $w^{-1} ( - \alpha_i + n \iota)$ is negative if $n - \kappa_i(w^{-1}) \leq 0.$ Hence $\wh{\Delta}_{w}$ from (\ref{Deltawdef}) includes the roots \be{flipped:w} -\alpha_i+ \iota, \,- \alpha_i + 2 \iota, \,\ldots, \,- \alpha_i + \kappa_i(w^{-1}) \iota\,. \ee
   From this and (\ref{lengthandcard}) we conclude that $\kappa_i(w^{-1})    \leq   \ell(w^{-1})  = \ell(w)$.
  A similar analysis for the case of $\sigma_i <0$ shows $\wh{\Delta}_{w}$ contains the string  \be{string:2} -\alpha_i + \iota,\, \ldots,\, - \alpha_i + (\kappa_i(w^{-1})-1)\iota, \ee and so
  $\kappa_i(w^{-1})    \leq    \ell(w^{-1})  +   1 $.
\end{proof}

\vspace{.5cm}
\begin{center} {\bf D. Loop Groups and Some Decompositions}\end{center}

\spoint \label{groups:section} We next introduce some notation related to loop groups following \cite{ga:ihes}, which we briefly review here. Let $\, \mf{g}_{\zee} \subset \mf{g}\,$ and $\,\hf{g}^e_{\zee} \subset \hf{g}^e\,$ be Chevalley $\zee$-forms of the Lie algebras defined above.   Given a dominant integral weight $\lambda \in \hf{h}^*$  extended to  $(\hf{h}^e)^*$  as in (\ref{integralweightlatticedef}-\ref{2.20}), let $V^{\lambda}$ be the corresponding irreducible highest-weight representation of  $\hf{g}^e$. Also,  let $V^{\lambda}_{\zee} \subset V^{\lambda}$ denote a Chevalley $\zee$-form for this representation, chosen compatibly with the Chevalley forms  $\mf{g}_{\zee}$ and $\hf{g}^e_{\zee}$ and with suitable divided powers as in \cite{ga:la}.

For any commutative ring $R$ with unit,   let  $\mf{g}_R$, $\hf{g}^e_R$, and  $V^{\lambda}_R$ be the objects obtained by tensoring the respective objects $\mf{g}_\Z$, $\hf{g}^e_\Z$, and  $V^{\lambda}_\Z$ over $\Z$ with $R$.  For any field $k$, we let $\hG^{\lambda}_k \subset \Aut(V_k^{\lambda})$ be the group defined in \cite[(7.21)]{ga:ihes} which generalizes the classical Chevalley-Steinberg construction of \cite{stein:yale} to the loop setting. When no subscript is present, we shall implicitly assume that $k= \R$, so that for example $\hg^{\lambda}$ is taken to mean $\hG^{\lambda}_{\R},$ etc.

For $a \in \wh{\Delta}_W$ and $u \in k$ we define $\chi_a(u)$ as in \cite[(7.14)]{ga:ihes},   which parameterizes the one-parameter root group corresponding to $a.$ Using $\chi_a(u)$ we may then define the elements
 \begin{equation}\label{wasdef}
    w_a(s) \ \ := \ \ \chi_a(s)\,\chi_{-a}(-s^{-1})\,\chi_a(s)
 \end{equation}
   and
   \begin{equation}\label{hasdef}
    h_a(s) \ \ := \ \ w_a(s)\,w_a(1)^{-1}
   \end{equation}
  for $s \in k^*$.   We shall use the abbreviation $h_i(s)=h_{a_i}(s)$ for $i=1,\ldots,\ell+1$.  We shall fix $\lambda$ throughout and shorten our notation to $\hG:= \hG^{\lambda}_{\R}$ and make a similar convention for $\hG_k.$

 In case $k= \R,$  we let $\{\cdot, \cdot \}$ denote the real, positive-definite inner product on $V^{\lambda}_{\R}$ described in \cite[\S16]{ga:ihes} and the references therein.   We then define subgroups
\begin{equation}\label{disc:cpt}
\aligned
 \hgam & \ \ = \ \ \left\{ \gamma \in \hG  \  |  \   \gamma \,  V_{\zee}^{\lambda} \,  = \,  V_{\zee}^{\lambda} \right\} \\
\text{and} \ \ \ \ \hK & \ \ = \ \ \left\{ k \in \hG \  | \  \{ k \xi, k \eta \} = \{ \xi, \eta \} \; \text{ for } \xi, \eta \in V^{\lambda}_\R \right\}.
\endaligned
\end{equation}
 Fix a coherently ordered basis $\mc{B}$ of $V_{\zee}^{\lambda}$ in the sense of \cite[p.~60]{ga:ihes},  and let $\hh$,  $\hU$, and $\hB$  respectively denote the subgroups of $\hG$ consisting of diagonal, unipotent upper triangular, and upper triangular  matrices with respect to $\mc{B}$.  The subgroup $\wh{U}$ contains
 the one-parameter subgroups $\chi_a(u)$ for each $a\in\wh{\D}_+$.
 The subgroup $\hh$ normalizes $\hu$ and
 \be{Bsemi} \hb \ \ = \ \  \hh  \, \ltimes \, \hu \, \ee
 is their semi-direct product.
   There exists a surjective map  from
   $(\R^*)^{\ell+1}$ onto $\hh$ given by  \be{new2.74} (s_1, \ldots, s_{\ell+1})  \ \ \mapsto  \ \ h_1(s_1) \, \cdots \,h_{\ell+1}(s_{\ell+1})\,. \ee
   We use this parametrization to  define the map $h\mapsto h^{a_i}$ on $\wh{H}$ for any simple root $a_i$, via the formula $h_{j}(s)^{a_i}=s^{\langle a_i,h_j \rangle}$.  The products  on the righthand side of (\ref{new2.74}) with each $s_i>0$ form a subgroup $\wh{A}\cong \R_{>0}^{\ell+1}$ of $\wh{H}$, so that $\wh{A}$ is isomorphic to $\wh{\frak h}$ by the logarithm map
   \begin{equation}\label{lndef}
   \ln \, : \, h_1(s_1)\cdots h_{\ell+1}(s_{\ell+1}) \ \ \mapsto  \ \ \ln(s_1)\,h_1 \ + \ \cdots \ + \ \ln(s_{\ell+1})\,h_{\ell+1}\,.
   \end{equation}
  Both $\wh{H}$ and $\wh{A}$ have Lie algebra $\wh{\frak h}$, and in fact
   \be{Hhat} \hh  \ \ = \ \  \ha \,\times \,(\hh \cap \hk)  \ee
   holds as a direct product decomposition.
       Using (\ref{lndef}) the notation (\ref{atolambdef}) extends to give a character $a\mapsto a^\l$ of $\wh{A}$ for any $\l\in(\wh{\frak h}^e)^*$.
    For any linear combination
 \begin{equation}\label{new2.75}
    X  \ \ = \ \  c_1 \, a_{1} \  + \  \cdots  \ + \  c_{\ell+1}\, a_{\ell+1} \, ,  \ \ \  \ \  c_i \, \in \, \zee \,,
 \end{equation}
   we define the map \be{new2.76} h_X(s)  \ \ := \ \  h_1(s)^{c_1}\, \cdots \, h_{\ell+1}(s)^{c_{\ell+1}} \ \ \ \ \ \text{for} \ \,   s \,\in\,\R^*\,.\ee  Using the relation $\iota=\a_0+a_{\ell+1}$ from (\ref{newroot}), this then allows us to define  $h_{\iota}(s)$  for $s \in \R^*$.

    Define
    \begin{equation}\label{hh}
    \aligned
        \hh_{cen} & \ \ := \ \  \left\{\, h_{\iota}(s) \,|\,  s \,\in\, \R^* \right\} \\
\text{and} \ \ \ \ \ \ \ \hh_{cl}\,  & \ \ :=  \ \ \left\{ \prod_{i=1}^{\ell} h_{i}(s_i) \ |\ s_1,\ldots,s_{\ell}\, \in\, \R^* \right\},
\endaligned
    \end{equation}
and then set $\ha_{cen} = \ha \cap \hh_{cen}$ and $\ha_{cl} = \ha \cap \hh_{cl}$. We then have a direct product decomposition \be{hh:dp} \ha \ \  = \ \  \ha_{cen}  \, \times \,  \ha_{cl}\,. \ee Indeed, it suffices to show that $\ha_{cl} \cap \ha_{cen}=\{1\}.$ As any element $x \in \ha_{cl} \cap \ha_{cen}$ can be written as
\be{x:2} x  \ \ = \ \  h_{\iota}(s) \ \ = \ \  h \,, \ \   \ \text{ with }  \ s \,  \in \,\R_{>0} \ \ \text{and} \ \ h \,  \in \, \ha_{cl}\, . \ee
Recall from above that the group $\wh{G}$ is defined with respect to a dominant integral weight $\l$.
Because of  \cite[Proposition 20.2]{ga:ihes}, there exists a positive integer $m$ together with a homomorphism $\pi(\lambda, m \Lambda_{\ell+1}): \hg^{\lambda} \hookrightarrow \hg^{m \Lambda_{\ell+1}}.$ According to \cite[p.~107]{ga:ihes} this map identifies the elements $h_i(s)$ in the two groups, for any $1\le i \le \ell+1$.
From (\ref{aff:wts}) and (\ref{newcoroot}) we see that the fundamental weight $\Lambda_{\ell+1}$ satisfies $\la \Lambda_{\ell+1}, h_{\iota} \ra =1$ and $\la \Lambda_{\ell+1}, h_{i} \ra = 0$ for $i=1, \ldots, \ell.$
In the  group $\hg^{m \Lambda_{\ell+1}}$, which acts on the highest weight representation $V^{m \Lambda_{\ell+1}}$ with highest weight vector $v_{m \Lambda_{\ell+1}}$,  \cite[Lemma~11.2]{ga:ihes} shows
 \be{Acen:Acl} h_{\iota}(s)\, v_{m \Lambda_{\ell+1}}  \ \ = \ \  s^{ m \la \Lambda_{\ell+1}, h_{\iota} \ra }\, v_{m \Lambda_{\ell+1}}  \ \ = \ \  s^m\, v_{m \Lambda_{\ell+1}}\,. \ee
 However, using (\ref{x:2}) the same lemma shows
$h_{\iota}(s)   v_{m \Lambda_{\ell+1}} =  h  v_{m \Lambda_{\ell+1}}   = v_{m \Lambda_{\ell+1}}$
 because  $h \in \ha_{cl}.$ Thus $s^m=1$ from which we can conclude that $s=1$ itself.

%
%
%
%

\spoint \label{bruhat:section} For any field $k$, we define $\hh_k$ to be the subgroup of $\hg_k$ consisting of diagonal matrices with respect to the  coherent basis $\mc{B}$. Similarly we set $\hb_k$ and $\hu_k$ to be the groups of upper triangular and unipotent upper triangular matrices with respect to $\mc{B}$, respectively. As in (\ref{Bsemi}), $\hb_k$ is the semi-direct product of $\hh_k$ and $\hu_k$.
Moreover, elements of  $\hh_k$ can be written in terms of the map (\ref{new2.74}) (recall that the elements $h_a(s)$ from (\ref{hasdef}) are defined for $s \in k^*$).

The group $\hg_k$ is equipped with a Tits system (see \cite[\S13-14]{ga:ihes}) which identifies the affine Weyl group $\aw$ from section~\ref{sec:What} as the quotient ${\mathbf N}_k/\hh_k =  \aw$, where ${\mathbf N}_k$ is the
group generated by the $w_{a_i}(s)$ from (\ref{wasdef}) with $1\le i \le \ell+1$ and $s\in k^*$.
Thus each  $w \in \wh{W}$ has a representative in ${\mathbf N}_k$, which we continue to denote by $w$.  Moreover, if $w$ is written as a word in the generators $w_1,\ldots,w_{\ell+1}$ from (\ref{sim:ref}), a representative can be chosen to be the corresponding word in the elements $w_{a_1}(1),\ldots,w_{a_{\ell+1}}(1)$ from (\ref{wasdef}).  We shall tacitly identify each $w\in \wh{W}$ with this particular representative.
In the special case $k  = \R$ these representatives lie in $\hk$.   With these conventions, there exists a Bruhat decomposition  \be{bruhat} \hg_k \ \  = \ \ \bigcup_{w \,\in \,\aw} \hb_k \, w \, \hb_k \,, \ee
where in fact each double coset $\hb_k w \hb_k$ is  independent of the chosen representative $w$.
Because of  (\ref{Bsemi}) and the fact that $\aw$ normalizes $\hh_k$,
every element   $g \in \wh{G}_k$  can  be written as
\begin{multline}\label{bruhat:elt} g \ \  = \ \  u_1 \,z \,w \,u_2\,,  \ \ \ \  \text{ where }  \ u_1,\,  u_2 \ \in \  \hu_k\, , \ \ w \ \in \ \wh{W}\,, \ \  \text{ and }   \ z  \ \in\ \hh_k\, .\end{multline}
If $k= \R$, the elements $w\in \wh{W}$ and $z\in \hh$ are uniquely determined by $g$, though $u_1$ and $u_2 \in \wh{U}$ are not in general.

\spoint \label{iwasawa:section}
Next, we recall from \cite[\S16]{ga:ihes} that there exist Iwasawa decompositions  \be{iwasawa} \hG  \ \ = \ \  \hU \, \hA\, \hK  \ \ = \ \  \hK  \, \hA  \, \hU \ee
with respect to the subgroups $\wh{U}$, $\wh{A}$, and $\wh{K}$ defined in section~\ref{groups:section},  with uniqueness of decomposition in either particular fixed order.  Given an element $g \in \hg$ we denote by $\iwa(g) \in \ha$ its projection onto the $\ha$ factor in the first of the  above   decompositions. Note that $\iwa: \hg \rr \ha$ is left $\hu$-invariant  and right $\hk$-invariant by construction.

  For $r \in \R,$ write  $s =e^{r}$ and define the exponentiated degree operator on $V^{\lambda}_\R$ by the formula \be{eta(s)} \eta(s)  \ \ = \ \  \exp(r \dd)\,. \ee
The element $\eta(s)$ acts on the one-parameter subgroups $\chi_{\a+n\iota}(\cdot)$ by
\begin{equation}\label{etaandchi}
\eta(s)\chi_{\a+n\iota}(u)\eta(s)^{-1} \ \ = \ \ \chi_{\a+n\iota}(s^nu)\,.
\end{equation}
  It furthermore acts as a diagonal operator with respect to the coherently ordered basis $\mc{B}$, and consequently it normalizes $\hU$ and commutes with $\hA.$
It then follows  that \be{ida:deg} \eta(s)\,\hG  \ \ = \ \  \{ \eta(s) g \,|\,g \,\in \,\hg \} \ \  = \ \  \hU  \,  \eta(s) \, \hA  \, \hK . \ee We then extend the function $\iwa$ above to a function  $ \iwea  : \eta(s) \hg   \rr   \eta(s) \ha $ by defining
\be{iweadef} \iwea( \eta(s)\, g )  \ \ = \ \  \eta(s)\, \iwa(g) \ \  \in  \ \ \eta(s) \,\ha\, .
 \ee
 This extension is obviously right $\hk$-invariant, and is also left $\hu$-invariant because $\eta(s)$ normalizes  $\hu.$    It furthermore satisfies
 \begin{equation}\label{iweaequivariant}
    \iwea( \eta(s)\, a_1\,u\,a_2\,k ) \ \ = \ \  \iwea( \eta(s)\, a_1)\,a_2 \ \ = \ \  \eta(s)\, a_1 \,a_2
 \end{equation}
 for any $u\in \wh{U}$, $a_1, a_2\in \wh{A}$, and $k\in \wh{K}$.

We extend the logarithm map $\ln:\wh{A}\cong\wh{\frak h}$ defined in (\ref{lndef}) to  $\eta(s)\wh{A}$ by the rule
\begin{equation}\label{logslicerule}
    \ln(\eta(s)\,a) \ \ = \ \ r \,\dd  \ +  \ \ln(a)  \ \ \  \ \text{for} \  \ a\,\in\,\wh{A}
\end{equation}
 (cf. (\ref{eta(s)})).
For the duration of the paper we make the important restriction to consider \emph{only} the case  that $r > 0.$   Note that the convention in \cite{ga:ihes}
is  to  also consider  $r>0$, although $\eta(s)$ from (\ref{eta(s)}) is instead parameterized there as $\eta(s)=\exp(-r\dd)$.  This switch is because we   work here   with the Iwasawa decomposition (\ref{iwasawa}) having $\wh{K}$ on the right, whereas in \emph{op.~cit}   $\wh{K}$ is on the left. 

\vspace{.3cm}
\begin{center} {\bf E. Adelic Loop Groups}\end{center}

\spoint In this section, we review some aspects of  adelic loop groups and their decompositions; further details can be found in  \cite{ga:ragh, ga:ms2}. Let $\mc{V}$ denote the set of finite places of $\Q$, each of which can be identified with a prime number $p$ and the $p$-adic norm  $| \cdot |_p$. For each $p \in \mc{V}$ the field $\Q_p$ is the corresponding completion of $\Q$   and has ring of integers $\Z_p.$  We write  $\mc{V}^e = \mc{V} \cup \{\infty \}$ for the set of all places of $\Q$, where the place $\infty$ corresponds to the archimedean valuation  $| \cdot |_{\infty}$\,, i.e.,
the usual absolute value on $\Q_{\infty} = \R.$  The adeles are defined as the ring
$
    \A = \prod'_{p\in\mc{V}^e} \Q_p,
$
where the prime  indicates the restricted direct product of the factors with respect to the $\Z_p$.  Likewise, the finite adeles $\A_f$ are the restricted direct product of all $\Q_p$, $p\in\mc{V}$, with respect to the $\Z_p$. For each $a= (a_p) \in \A$ the adelic valuation is defined as  $|a|_{\A}    :=  \prod_{p \in \mc{V}^e} | a_p |_p$.
We also write $\I=\A^*$ for the group of ideles and $\I_f=\I\cap \A_f$ for the group of finite ideles.

In section~\ref{groups:section} we introduced the exponentiated group $\hg_k$ for any field $k$, in particular $k=\Q_p$ for any  $p\in\mc{V}^e$.  For shorthand denote $\hg_p \ :=  \ \hg_{\Q_p}$, so that $\hg_{\infty}$ is just the real group $\hg$. For each $p \in \mc{V},$ we set
 \be{Kp} \hk_p  \ \ = \ \  \{ \, g \in \hg_p \,\mid \,  g V^{\lambda} _{\zee_p} \, =\, V_{\zee_p}^{\lambda} \,\} .\ee By convention, we also set $\hk_{\infty}$ to be the group  $\hk$  introduced in (\ref{disc:cpt}). The adelic loop group is then defined as
 \be{Gad} \hg_{\A}  \ \ := \ \  \prod'_{p \,\in \, \mc{V}^e} \hg_p\,, \ee
 where the product is restricted with respect to the family of subgroups $\{ \hk_p \}_{p \in \mc{V}}.$  We also define
 \be{Kad} \hk_{\A} \ \  = \   \prod_{p \, \in \, \mc{V}^e} \hk_p  \ \ \subset \ \  \hg_\A\,. \ee
 Analogously, the groups $\hg_{\A_f}$ and $\hk_{\A_f}$ are defined
 by replacing $\mc{V}^e$ with $\mc{V}$ in (\ref{Gad}) and (\ref{Kad}), respectively.

We set $\hh_p:= \hh_{\Q_p} \subset \hg_{p}$ to be the group defined at the beginning of \S\ref{bruhat:section}, where we remarked that it is generated   by the elements $h_1(s)$, $h_2(s)$,\ldots, $h_{\ell+1}(s)$ for  $s \in \Q_p^*$;  thus for example, $\hh_{\infty}=\hh$. Define
$ \had  =  \prod'_{p  \in   \mc{V}^e} \hh_{p}$,
 where the product is restricted with respect to the family of subgroups $\{ \hh_{p} \cap \hk_p \}_{p \in \mc{V}}.$ Analogously to (\ref{new2.74}), every element $h \in \had$ has an expression
 \be{h:ad}
 h  \ \ = \   \   \prod_{i\,=\,1}^{\ell+1} h_i(s_i)\,, \ \ \  \text{ where each }\, s_i  \, \in \,\I\,.
 \ee
 For such an expression, we define its norm to be the element of $\ha$ given by the product
 \be{h:norm}
 |h|  \ \ = \ \  \prod_{i \,=\,1}^{\ell+1} h_i( |s_i|_\A)\,,
 \ee
 which can be shown to be  uniquely determined by $h \in \had$ independently of its factorization (\ref{h:ad}).  We set 
 $\uad =  \prod'_{p   \in    \mc{V}^e} \hu_{\Q_p}$,
 where the product is restricted with respect to the family $\{ \hu_{\Q_p} \cap \hk_p \}_{p \in \mc{V}}.$ We shall also write 
 $\bad  = \uad \cdot \had$,
which itself is the restricted direct product of all $ \hu_{\Q_p} \cdot  \hh_{p}$ with respect to their intersections with $\wh{K}_p$.

\spoint \label{adelic:iwasawa} We next state the adelic analogues of the Iwasawa decompositions (\ref{iwasawa}) and (\ref{ida:deg}).  First, for each $p \in \mc{V}$ there is the $p$-adic Iwasawa decomposition $\hg_{\Q_p} = \hu_{\Q_p}   \hh_{\Q_p}   \hk_{\Q_p}$,
which is not a direct product decomposition because $\hh_{\Q_p} \cap \hk_{\Q_p}$ is nontrivial.  Together with the $p=\infty$ decomposition (\ref{iwasawa}), these local decompositions give   the adelic Iwasawa decomposition  \be{ad:iwasawa} \gad  \ \ = \ \  \uad  \, \had \, \kad\,. \ee
Although the adelic Iwasawa factorization
\be{ad:iwa:elta} g \ \  =  \ \ u_g \, h_g \, k_g\,, \ \ \ \  \ \ u_g \,\in\, \uad, \, h_g \,\in \,\had, \, k_g \,\in\, \kad\,, \ee is not in general unique, the element $|h_g|$ defined using (\ref{h:norm})  is  uniquely determined by $g$. We shall write the projection onto this element as the map
\be{ad:iwa:1}
\aligned
\aiw_{\ha} \ :    \ \gad  \ &  \rr  \ \ha \, , \\
 g  \ & \mapsto \ |h_g| \,
\endaligned
\ee
onto the Iwasawa $\wh{A}$-factor of $\wh{G}$.    Note that $|h_g|$ is an element of the real group $\wh{G}$; we do not adelize $\wh{A}$.

Recall that we have defined $\eta(s)$ in (\ref{eta(s)}) in the context of real groups, where it has a nontrivial action on $\wh{G}=\wh{G}_{\Q_\infty}$ by conjugation.  After extending this action trivially to each $\wh{G}_{\Q_p}$, $p\in \mc{V}$, there is then a twisted Iwasawa decomposition  \be{twist:iwa:ad} \eta(s)\,\gad \ \  =  \ \ \uad  \, \eta(s) \, \had \,\kad\,. \ee
Moreover, if we write $\eta(s) g \in \eta(s)\gad$ with respect to the above decomposition as \be{ad:iwa:eltb} \eta(s) \, g  \ \ = \ \  u_g \, \eta(s) \, h_g \, k_g\,, \ \  \ \ \ \  u_g \, \in\,  \uad, \, h_g \, \in \, \had, \, k_g \, \in \, \kad\,, \ee then the element $|h_g|$ from (\ref{h:norm})  is again uniquely determined. We then have the projection
\be{ad:iwa:2}
 \aligned
\aiw_{\eta(s) \ha} \ :    \ \eta(s) \,\gad  \ &  \rr  \ \eta(s)\, \ha \, , \\
 g  \ & \mapsto \  \eta(s) \,|h_g| \,,
\endaligned \ee
generalizing (\ref{iweadef}).

\spoint\label{sec:2.15} In \S\ref{groups:section} the group $\hg_\Q$ was defined over the field $k=\Q$. It has embeddings $\hg_{\Q} \hookrightarrow \hg_{\Q_p}$  for each $p \in \mc{V}^e$, and hence a diagonal embedding
 \be{diag} i \ : \  \hg_{\Q}   \ \ \hookrightarrow  \ \ \prod_{p \, \in\, \mc{V}^e} \hg_{\Q_p} \,.\ee
 Note that the righthand side is the direct product, not the restricted direct product:~this is because
  $i(\hg_\Q)$ is actually not contained in $\gad$ (see  \cite[\S2]{ga:ragh}), since  in contrast to the finite-dimensional situation, an element of $\wh{G}_\Q$ can involve different prime denominators in each of infinitely many root spaces.     Define
  \be{gamma:Q} \hgam_{\Q} \ \ := \ \  i^{-1} (\gad)  \ \ \subset  \ \ \hg_\Q\ \ee
   as the subgroup of $\hg_\Q$ which does embed into $\gad$.  We shall generally follow the common convention of identifying $\wh{\G}_\Q$ with its diagonally embedded image $i(\wh{\G}_\Q)$.

\section{Iwasawa Inequalities} \label{section:iwasawa}


\spoint In (\ref{Deltawdef})
 we  associated  to each $w \in \wh{W}$ the finite set of roots
 \be{re:wflip} \wh{\Delta}_{w^{-1}} \ \  = \ \ \{a \in \wh{\Delta}\,|\,a>0, \ w a<0\}\,.\ee
 Recall that the Weyl group fixes all imaginary roots, so $\wh{\Delta}_{w^{-1}}\subset  \wh{\Delta}_W$.
  For each root $a \in \wh{\Delta}_W$  let $U_a$ denote the corresponding   root group consisting of the elements $\{ \chi_a(s) | s \in \R \}$, and fix any simple order (i.e., a partial order in which every two elements are comparable) order on $\wh{\Delta}$. Let  $\wh{U}_{-}$ be the subgroup of $\wh{G}$ which acts by unipotent lower triangular matrices on the coherent basis $\mathcal B$ from section~\ref{groups:section}. We then have the \emph{subgroups}
\be{Uw}
 U_{w} \ \  = \ \  \prod_{\ \  a \,  \in \,\wh{\Delta}_{w^{-1}}} \!\!\! \! U_{a}  \ \ = \ \  \wh{U}\,\cap\,w^{-1}\wh{U}_{-} \, w  \ \ \subset \ \ \wh{U}
  \ee
  and
   \begin{multline}\label{U-w} U_{-, w} \ \ = \ \  w \, U_{w} \, w^{-1} \ \ = \ \ \prod_{  \ \ \ \ \gamma\, \in \,\wh{\Delta}_{-,w^{-1}}}\!\!\!\!\!\! U_{\gamma}   \ \  = \ \  \wh{U}_{-}\,\cap\,w\wh{U}w^{-1} \ \  \subset \ \ \wh{U}_{-}\,, \end{multline}
   where the product is taken with respect to the fixed order on $\wh{\D}$ and one has uniqueness of expression
   (cf.~\cite[Lemma~6.4 and Corollary~6.5]{ga:lg2}).

Similarly, each $U_a$ has a rational subgroup  $U_{a,\Q}=\{ \chi_a(s) | s \in \Q \}$.  Rational subgroups  $U_{w, \Q}\subset U_w$ and  $U_{-, w ,\Q}\subset U_{-,w}$ are defined as products of  $U_{a,\Q}$ and $U_{\g,\Q}$  over the roots $a$ and $\g$ appearing in (\ref{Uw}) and (\ref{U-w}), respectively.  Because $\chi_a(s)$ is also defined for $s\in \Q_p$, we likewise have subgroups $U_{a,\Q_p}$, $U_{w,\Q_p}$, and $U_{-,w,\Q_p}$, and their adelic variants $U_{w,\A}$ and $U_{-, w, \A}$ defined as  restricted direct products.

\spoint Recall the notation for the Iwasawa decomposition and its adelic variant introduced in (\ref{iweadef}), (\ref{ad:iwa:1}), and (\ref{ad:iwa:2}).  For  elements $x,y$ in any group we shall use the shorthand notation $x^y := y x y^{-1}$.

\begin{nlem}\label{lemma:iwasawadecomp}
i) Let $\gamma_\Q \in \hgam_{\Q}$, regarded as diagonally embedded in $\wh{G}_\A$ as in \S\ref{sec:2.15}, and  let $g \in \hg.$ Use (\ref{bruhat:elt}) to write $\gamma_\Q = u_1 z w u_2 $ with $u_1, u_2 \in \hu_{\Q},\, w \in \aw,$ and $z \in \hh_{\Q}.$ Regard the element $\gamma_\Q  \eta(s) g \in \eta(s) \gad$ and define the element $\aiw_{\eta(s) \ha}(\gamma_\Q \eta(s) g)  \in \eta(s) \ha$ as in (\ref{ad:iwa:2}).
Then  \begin{equation}\label{4:2:-1a}
\aiwea(\,\g_\Q \,\eta(s)\, g\,)  \ \ = \ \  \iwea( \eta(s) g)^w \,\cdot\,   \aiwa(w u_w)
\end{equation}  for some element $u_w \in U_{w, \A}$ (depending on $\gamma_\Q$ and $\eta(s) g$).

ii) If $\g\in \wh{\G}\subset \wh{G}$,
  \begin{equation}\label{4:2:-1b}
\iwea(\,\g \,\eta(s)\, g\,)  \ \ = \ \  \iwea( \eta(s) g)^w \,\cdot\,   \aiwa(w u_w)
\end{equation}
 for some element $u_w \in U_{w, \A}$ (again depending on $\gamma$ and $\eta(s) g$).
 \end{nlem}
  \begin{proof}
 First we note that since $z \in \hh_{\Q}$ we have that $\aiwa(z)=1.$ Using the left $\uad$-invariance, the commutativity of $\wh{A}$ with $\eta(s)$,  and the right $\kad$-invariance of $\aiwea$, we compute
    \begin{equation} \label{gammag:-1} \aligned \aiwea( \, \gamma_\Q  \, \eta(s) \, g \, )   & \ \ = \ \ \aiwea(  \, u_1 \,  z \,  w  \, u_2 \,  \eta(s) \,   g \,   ) \\ & \ \ = \ \ \aiwea(\,  z \,  w  \,  u_2 \,  \eta(s) \,  g \, ) \\ & \ \ = \ \ \aiwea(\,  w  \,  u_2 \,  u'\,\eta(s)\,h' \, )\,,
   \endaligned
     \end{equation}
 where
 \begin{multline}\label{4.2:0}
  \eta(s) \, g \ \ = \ \ u'\, \eta(s)\,h'\,k'\ ,  \ \ \ \text{ with }  \ u' \,\in\, \hu \,, \ h' \,=\, \iwa(g) \, \in\, \hA\,, \ \,  \text{and} \, \ k' \, \in \,  \hK\,,
 \end{multline}
 is the (real, and hence also adelic) Iwasawa decomposition of $\eta(s)g$.  Let $u= u_2 u' \in \uad.$
   Again using these properties plus (\ref{iweaequivariant})  we then have
 \begin{equation}\label{gammag:-2}
    \aligned
    \aiwea( \, \gamma_\Q  \, \eta(s) \, g \, )    \ \ & = \ \
    \aiwea( w\,u\,\eta(s)\,h' \, )\\
     & = \ \   \aiwea(  \,  w \, \eta(s)\,h'\, w^{-1}\,\cdot\,w\, \cdot\,(\eta(s)h')^{-1}\,\cdot\,u\,\cdot\,(\eta(s)h')\,
     ) \\
     & = \ \
   \ \aiwea(\,
     (\eta(s) h')^w\,\cdot\,w\,u''\,
     ) \\
     & = \ \ (\eta(s) h')^w \,  \iwa^\A(\,w\,u''\,
     )
  \,,
    \endaligned
 \end{equation}
where $u''=(\eta(s)h')^{-1}u(\eta(s)h')$ lies in $\uad$ because $\eta(s)$ and $h'\in \widehat{A}$ both normalize $\uad$. Furthermore, $u''$ can be factored as $u''=u'''{u}_w$, where $wu'''w^{-1}\in \uad$ and ${u}_w\in U_{w,\A}$.
Thus $wu''=(wu'''w^{-1})w u_w$ and
formula (\ref{4:2:-1a}) now follows from once more applying the left $\uad$-invariance of $\aiwea$.  This proves part i).

To prove part ii), let $\g_\Q\in \wh{\G}_\Q$ be the diagonal embedding of $\g$ into $\wh{G}_\A$.  Since $\eta(s)g\in \wh{G}$, the archimedean component of $\g_\Q\eta(s)g$ is $\g\eta(s)g$.  At the same time, its nonarchimedean components are $\g_p$.  These lie in $\wh{K}_p$ since they preserve the lattice $V_{\Z_p}^\l$ (as a consequence of the defining fact that $\g\in \wh{G}$ preserves $V_\Z^\l$).  Hence the projection of $\g_\Q\eta(s)g$ onto $\wh{G}_{\A_f}$ lies in $\wh{K}_{\A_f}$, and by definition  (\ref{ad:iwa:1}) the left hand sides of (\ref{4:2:-1a}) and (\ref{4:2:-1b}) agree.
 \end{proof}%

In particular part ii) of the lemma asserts that for $\g\in \wh{\G}$,
\begin{equation}\label{hpartandlogs}
    \ln \(\iwea\,(\,\g \,\eta(s)\, g\,)\) \ \ = \ \  \ln \(\eta(s)^w\)  \ + \  \ln\(\iwa( g)^w\)  \ + \ \ln \(\aiwa(\,w\, u_w\,)\)
\end{equation}
for some element $u_w\in U_{w,\A}$, where the logarithm map $\ln$ was defined in  (\ref{lndef}) and (\ref{logslicerule}).

\spoint We next establish some estimates on the individual terms in (\ref{hpartandlogs}) for later use in section~\ref{section:convergence}.   For any parameter $t>0$ let
 \be{At}
  \ha_t  \ \ := \ \  \{ \,h \in \ha \ |\  h^{a_i} > t \; \text{ for each }  \  i\,=\,1,\, \ldots, \,\ell+1 \,\}\,,
   \ee
   and let  $\hu_{\mc{D}}$ be a fundamental domain for the action of $\hgam \cap \hu$ acting on $\hu$ by left translation.  Siegel sets for $\eta(s)\hat{G}$  were defined in \cite{ga:ihes} as sets of the form
 \be{loopsiegel}
 \hs_t \ \  := \ \  \hu_{\mc{D}}  \,    \eta(s)\,  \ha_t\,  \hk\,,
 \ee
 for some choices of $t>0$ and $\hu_{\mc{D}}$.

 According to (\ref{newcoroot}) and (\ref{extended:cartan}), any element $X \in \hf{h}$ can be decomposed as
  \begin{multline}\label{hhat:decomp}
   X  \ \ = \ \  X_{cl} \  + \   \la \lal, X \ra h_{\iota} \  , \ \ \  \  \ \text{with}  \ \  X_{cl}  \, \in \,  \mf{h} \ \ \text{and}  \  \ \la \lal , X \ra \,\in\,\R\, ,
  \end{multline}
where we recall that $\lal:\wh{\frak h}\rightarrow \R$ from (\ref{aff:wts}) is the ($\ell+1$)-st fundamental weight (it is trivial on all classical coroots).
Note that because $h_\iota$ and $\cc$ are nonzero multiples of each other, properties (\ref{2.28}) and (\ref{omega:ext}) assert that all classical roots $\a\in \Delta$,  fundamental weights $\omega_1,\ldots,\omega_\ell$, and $\rho$ from (\ref{rho}) extend trivially to $\R h_\iota\oplus \R\dd$. Also, the classical coroots may be identified with elements of $\wh{\frak h}^e$ using the remarks after (\ref{omega:ext}).
The classical component $X_{cl}$ can be expanded as a linear combination  \be{Xcl} X_{cl} \ \ = \ \  \sum_{j\,=\,1}^\ell \, \la \omega_j, X \ra \, h_j \ee   of the  coroot basis $\{h_1,\ldots,h_{\ell}\}$ of $\frak h$ using (\ref{omega:ext}).
Because   of property (\ref{rho:h_i}) we   have that \be{norm:H} \langle \rho , X    \rangle \ \  = \ \ \langle \rho , X_{cl}    \rangle \ \  = \ \   \sum_{j\, =\,1}^\ell \, \la \omega_j, X \ra \,. \ee
With the above notation, we can now state the following result, which is one of the   main technical tools in this paper.


\begin{nthm} \label{iwineq} Fix an element $\eta(s)g$ of a Siegel set $\hs_t \subset  \eta(s)\wh{G}$ as in (\ref{loopsiegel}).
Recall from (\ref{eta(s)}) that $\eta(s)=\exp(r \dd)$ for $s=e^r$, $r\in \R_{>0}$.
Then there exist positive constants $E =E(s,t,\D),$ $C_1=C_1(s, t, \Delta),$ $C_2(s, t, \Delta)$ depending only on $s$, $t$, and the underlying classical root system $\Delta$, and another positive constant $C_3=C_3(g, s, t, \Delta)$ depending only on $g$, $s$, $t$, and $\D$ with the following properties:~for any   element  $w \in W^{\theta}$ as in (\ref{kostant})  and  $u_w\in U_{w, \A}$, the vectors  $H_1, H_2,$  and $H_3 \in \wh{\frak h}$ defined by
\begin{equation}\label{iwineqH1H2def}
    H_1 \ \ := \ \ \ln(\,\aiwa(wu_w)\,)\ ,\ \ \ \    \ \ H_2 \ \ := \ \
    \ln(  \eta(s)^w  ) \ - \ r \,\dd\ , \ \ \ \ \text{and} \ \  \ \  H_3 \ \ := \ \  \ln ( \iwa(g)^w)
\end{equation}
satisfy the inequalities
\begin{equation}\label{iwineq1}
    \la \omega_j, H_1 \ra \ \ \ge \ \ 0 \ \ \ \ \text{and} \ \ \ \ \la \omega_j, H_2 \ra  \ \ \ge \ \ 0 \ \ \ \  \ \ \ \ \ \  \text{for} \, \ j\,=\,1,\,\ldots,\,\ell\,,
\end{equation}
\be{iwineq3} | \la \omega_j, H_3 \ra | \ \ \le \ \ E \ \ \ \  \ \ \ \ \ \  \text{for} \, \ j\,=\,1,\,\ldots,\,\ell\,, \ee and
\begin{equation}\label{iwineq2}
\aligned
    \la \Lambda_{\ell+1} , H_1 \ra  \ \ & \geq \ \  - \,C_1 \ ( \ell(w) + 1)  \ \la   \rho  , H_1  \ra\,,\\
     \la \Lambda_{\ell+1}, H_2 \ra \ \ & \geq \ \  - \,C_2\  (\ell(w) +1 ) \  \la \rho , H_2   \ra\,,  \\   \text{and} \ \ \ \ \ \ \ \  \ \ \
     \la \Lambda_{\ell+1}, H_3    \ra  \ \  & \geq \ \ -\,C_3 \ ( \ell(w) +1 ) \,.
\endaligned
\end{equation}

  \end{nthm}
\noindent
{\em Remark:} In fact, one can choose $C_1=1$ in the simply-laced case (as the proof below will show).  The dependence of $C_3$ on $g$ is relatively mild, and only enters through $\langle \L_{\ell+1},\ln(\iwa(g))\rangle$ in (\ref{new3.43}).  In particular it is locally uniform in the central component of $g$.

\spoint \label{sec:3.4} The proof of the above inequalities will occupy
the  rest of this section.  We shall first draw a corollary which will   be   useful in the sequel. Recall from (\ref{hh:dp}) that the group $\ha$ has a direct product decomposition
\be{new3.18} \ha  \ \ = \ \   \ha_{cen} \ \times \  \wh{A}_{cl}\,, \ee
 where  $\ha_{cen}$ is the connected component of the one-dimensional central torus of $\hg$  and $\wh{A}_{cl}$ is the connected component of the split torus from the finite dimensional group $G$ underlying $\hg$.
  As before we identify $\mathbb{R}_{> 0}$ with $\ha_{cen}$ via the map $s \mapsto h_{\iota}(s)$. There is a  natural projection from $\eta(s) \ha$ onto $\wh{A}$,  and subsequently onto each of the factors $\ha_{cen}$ and $\ha_{cl}$ in the decomposition (\ref{new3.18}).

\begin{ncor} \label{cor:iwineq:2} There exist constants $C, D > 0$, depending only on $s$, $t$, the underlying classical root system $\Delta$, and locally uniformly in the central component of $g$, with the following property:~for any $w\in W^\th$, $\g\in \hgam \cap \wh{B} w \wh{B}$ and   $\eta(s) g \in \hs_t$ we have the estimate  
\be{main:ineq}  a^\rho  \ \ \geq \ \  D \, y^{-(C(\ell(w)+1))^{-1}}, \ee
where
 $a \in \ha_{cl}$ and $y \in \ha_{cen} \cong \R_{>0}$  are the projections of
\begin{equation}\label{coriwineq1}
   \iwea(\,\g\,\eta(s)\,g\,)  \ \ \in  \ \ \eta(s) \ha
\end{equation} onto $\ha_{cl}$ and $\ha_{cen}$, respectively.
 \end{ncor}
\begin{proof}

%
%
We start by
 using part ii) of Lemma~\ref{lemma:iwasawadecomp} (in particular, its consequence (\ref{hpartandlogs}))   to write the logarithm   of (\ref{coriwineq1}) as $rD+H_1+H_2+H_3$, where $H_1$, $H_2$, and $H_3$ are defined   in (\ref{iwineqH1H2def}).
This allows us to factor
\be{3.21} \begin{array}{lcr}   a^\rho   \ \ = \ \  x_1  \, x_2 \,  x_3  & \text{ and } & y \ \ = \ \  y_1\, y_2\, y_3\,, \end{array} \ee where    $x_i = e^{\la \rho, H_i \ra}$ and $y_i = e^{\la \lal, H_i \ra }$ for $i=1,2, 3$.

 From the inequalities (\ref{iwineq2}), there exists a constant $C>0$  depending only on $s, t, \Delta$, and $g$ (in whose central component it is locally uniform) such that
\begin{equation}\label{new3.22}\
\aligned
 y_1 & \ \  \ge \ \  x_1^{-C(\ell(w)+1)}\,, \\ y_2 & \ \  \ge \ \  x_2^{-C(\ell(w)+1)}\,, \\
  \text{and} \ \ \ \ \ \ \ \ \  y_3 & \ \  \ge  \ \ e^{-C(\ell(w)+1)}\,. \endaligned
\end{equation}
Thus
\begin{equation}\label{3.22andahalf}
    y^{-(C(\ell(w)+1))^{-1}} \ \ = \ \ (y_1\,y_2\,y_3)^{-(C(\ell(w)+1))^{-1}}  \ \ \le \ \ x_1\,x_2\,e \ \ = \ \   a^\rho   \,\smallf{e}{x_3}\,.
\end{equation}
Because of (\ref{rho}) and  (\ref{iwineq3}), the ratio $\f{e}{x_3}$ is bounded above and below by  positive constants depending only on $s$, $t$, and $\Delta$, establishing (\ref{main:ineq}).

\end{proof}

%


\spoint In this subsection we prove the first inequalities of (\ref{iwineq1}) and (\ref{iwineq2}).
We begin by recalling from \cite[Lemma~6.1]{ga:ms2} that for any  $w \in W^{\theta}$ and  $u_w  \in U_{w, \A}$,
\begin{equation}\label{arthur} H_1 \ \ = \ \
    \ln(\aiwa(wu_w))
    \ \  =   \sum_{ \ \ \ \ \ \ \ \g\,\in\,\Delta_{-, w^{-1} }} c_\g \,h_{\gamma}   \ \ \ \ \
    \text{with} \ \ c_\g \,\ge\,0\,,\end{equation}
    where $\Delta_{-, w^{-1} }=\{\g\in \wh{\D}_{-}|w^{-1}\g>0\}$ was introduced in (\ref{neg:flipped}).
     (Note that the equivalent statement in \cite[Lemma~6.1]{ga:ms2}  uses the $\wh{G}=\wh{K}\wh{A}\wh{N}$ Iwasawa decomposition as opposed to the $\wh{G}=\wh{N}\wh{A}\wh{K}$ Iwasawa decomposition used here.)
    Recall from the definition (\ref{kostant}) of $W^\theta$ that $\g$ cannot be the negative of a simple classical root $\a_1,\ldots,\a_{\ell}$, nor  in their span.  Since $\wh{W}$ acts trivially $\wh{\D}_I$, $\Delta_{-, w^{-1} }$ is a subset of $\wh{\D}_W$ and its elements have the form (\ref{deltawhat}).

\begin{nlem} \label{rtsystem} Assume that $w \in W^{\theta}$ and $ \g \in \Delta_{-, w^{-1} }$.  Then $\g$ has the form $\g = \beta + n \iota$ with $\beta \in \Delta_+$ and $n <0$.\end{nlem}

\begin{proof}  Since $\g<0$ we must have that $n\le 0$, and in fact $n<0$ since it is not in the span of    $\a_1,\ldots,\a_{\ell}$.  Let us write $\beta$ as an integral linear combination $ \beta = \sum_{i=1}^\ell d_i \alpha_i$ of the positive simple roots.    If  $d_1,\ldots,d_\ell \leq 0$ then \be{iw1:lemma} w^{-1} \g \ \  = \ \  \sum_{i=1}^\ell \, d_i \  w^{-1} \alpha_i  \ + \  n \iota \ee
exhibits the positive root $w^{-1}\g$ as a nonpositive integral combination of positive roots, a contradiction.  Thus at least one (and hence all)  $d_i > 0$ and $\b$ is a positive root. \end{proof}

If we write  $\gamma = \beta + n \iota\in \Delta_{-, w^{-1}}$ with $\beta \in \Delta_+$ and $n<0$ as in the lemma, then   (\ref{h_a:h'}) asserts the existence of a  positive constant $m$ such that $h_{\gamma} =  h_{\beta} + m n h_{\iota}.$  It follows from (\ref{arthur}), (\ref{new2.1}), (\ref{omega:ext}), (\ref{2.39}-\ref{2.40}), and the lemma that  $\la \omega_j ,H_1 \ra \geq 0$ for $j=1, \ldots, \ell$, which proves the first inequality of (\ref{iwineq1}).  Expand $H_1=\sum_{j=1}^\ell \la \omega_j, H_1 \ra h_j+\la \Lambda_{\ell+1}, H_1 \ra h_\iota$ as in (\ref{hhat:decomp}-\ref{Xcl}), so that
\be{wH} w^{-1}H_1 \ \  = \ \  \sum_{j\,=\,1}^\ell \la \omega_j , H_1 \ra \,h_{w^{-1} a_j}  \ + \  \la \lal, H_1\ra \,h_{\iota } \ee
by (\ref{gencorootdef}).
 For each $j=1, \ldots, \ell$ we write the positive root $w^{-1} a_j$ as $ \beta(j) + \kappa_j(w^{-1}) \iota$, with $\beta(j) \in \Delta$ and $\kappa_j(w^{-1})  \geq 0$. Hence, again from (\ref{h_a:h'}) there exists a positive constant $m_{w^{-1}a_j}$ (which equals $1$ in the simply-laced case, and which can only take on a finite number of values in general), such that  \be{hw} h_{w^{-1} a_j} \ \ = \ \   h_{\beta(j)} \  +  \ m_{w^{-1}a_j} \, \kappa_j(w^{-1}) \, h_{\iota}\,. \ee
 The coefficient of $h_{\iota}$ in (\ref{wH}) is  then given by the sum
  \be{3.23.5} \la \Lambda_{\ell+1}, H_1 \ra \  + \  \sum_{j\,=\,1}^\ell m_{w^{-1}a_j} \, \la \omega_j, H_1 \ra \, \kappa_j(w^{-1}) \,. \ee On the other hand, from (\ref{weyloncoroot}), (\ref{arthur}), and the definition of $\gamma \in \Delta_{-, w^{-1}}$ it follows that $w^{-1}H_1$ is a  nonnegative integral combination  of positive coroots and so (\ref{3.23.5}) is nonnegative, in particular  \be{m:bound}  - \, \sum_{j\,=\,1}^{\ell} m_{w^{-1}a_j} \, \la \omega_j , H_1 \ra \, \kappa_j(w^{-1})
   \ \ \le \ \ \la \Lambda_{\ell+1} , H_1 \ra  \ \  \leq   \ \ 0
  \, . \ee
  The second inequality here comes from (\ref{arthur}), the lemma, and (\ref{h_a:h'}).
 Lemma~\ref{kap:ineq} asserts that $\kappa_j(w^{-1}) \le   \ell(w)+1.$  Since the values $m_{w^{-1}a_j}$ are bounded by an absolute constant (and again all equal to $1$ in the simply-laced case), the first inequality of (\ref{iwineq2}) now follows from (\ref{norm:H}) and the first inequality of (\ref{iwineq1}).

\spoint Next, let us turn to the second of the inequalities listed in (\ref{iwineq1}) and (\ref{iwineq2}). Recall that we have defined $H_2$ in (\ref{iwineqH1H2def})  through the   formula
\be{m2:conj} w \,(r \, \dd)  \ \ = \ \  r\, \dd  \ + \  H_2\,. \ee
Now, let us write $w = \t{w} T_b$  with $\t{w} \in W$ and $b \in Q^{\vee}.$
Decomposing $H_2=H_{2,cl}+ \la \Lambda_{\ell+1}, H_2 \ra h_{\iota} $ as in (\ref{hhat:decomp}), formula (\ref{affw:actionandcenter}) implies
\be{iw2:1} \begin{array}{lcr}  \la \Lambda_{\ell+1}, H_2 \ra  \ \ = \ \  -r\,\f{(b, b)}{2}  & \text{ and } &  H_{ 2,cl} \ = \  -r  \  \t{w}\,b\,. \end{array} \ee
We claim that
 \be{pos:rts}
   d_j \ \ := \ \  \la a_j , \t{w}\,b\ra \ \  \leq \ \  0 \ \ \ \  \text{ for } \ j\,=\,1,\, \ldots,\, \ell\,
.   \ee Indeed, from   (\ref{Tbdef})  we have that \be{iw2:2}
w^{-1} a_j  \ \ = \ \  T_{-b} \t{w}^{-1}a_j \ \ = \ \  \t{w}^{-1}a_j  \, - \,   \la \t{w}^{-1}a_j, b \ra \, \iota
\,. \ee By definition (\ref{kostant}),   $w^{-1}a_j >0$ for $w\in W^\theta$ and $j=1,\ldots,n$,  hence
 \be{new3.35} - \la \t{w}^{-1}a_j, b \ra  \ \ \geq  \ \ 0\,, \ee
i.e.,
 $d_j = \langle a_j,\t{w} b\ra =  \la \t{w}^{-1}a_j, b \ra \leq 0$ because of (\ref{affw:hact}).

  By Lemma~\ref{ragh} the inequality (\ref{pos:rts}) implies that    \be{pos:cowts}
   q_j \ \ := \ \  \la \omega_j , \t{w}\,b  \ra \ \  \leq \ \  0 \ \ \ \  \text{ for } \ j\,=\,1, \ldots, \ell\,
.   \ee   Hence,  \begin{multline}\label{pos:cowts:2}
   \la \omega_j , H_2 \ra \ \ = \ \  \la \omega_j , H_{2,cl}  \ra \ \ = \ \  - r \,\la \omega_j , \t{w}\,b  \ra \ \  \geq \ \  0 \ \ \ \  \text{ for } \ j\,=\,1, \ldots, \ell\,
.   \end{multline} This proves the second inequality in (\ref{iwineq1}).

Let us now turn to the second inequality of (\ref{iwineq2}).  From (\ref{pos:rts}) and (\ref{pos:cowts}) we have the simultaneous expressions
\begin{equation}\label{wb:f}
\aligned  \t{w}\, b  \ \ & \ \ = \ \  \ \  \sum_{j=1}^\ell\, q_j\, h_j \ \ \ \ \text{ with } q_j  \ \ \leq  \ \ 0\\ \text{and} \ \ \ \ \ \ \ \ \ \
\t{w}\, b  \ \ & \ \ = \ \  \ \  \sum_{j=1}^\ell\, d_j\, \omega^{\vee}_j \ \ \ \ \text{ with } d_j  \ \ \leq  \ \ 0\,,  \endaligned
\end{equation}
where $\omega_j$ and $\omega_j^{\vee}$ are  defined in (\ref{new2.1})-(\ref{new2.2}). Using (\ref{cowt:corts}) we find that  \be{m2:1} (\t{w}b, \t{w}b) \ \  = \ \  \sum_{i\,=\,1}^\ell \smallf{2}{(\alpha_i, \alpha_i)}\, d_i \, q_i. \ee
Since the denominators take on only a finite number of positive values, there exists a  constant $B> 0$ such that
\be{m2:1.1} (\t{w}b, \t{w}b)
  \ \ \leq  \ \ B \, \sum_{i\,=\,1}^\ell d_i \, q_i
\ \ \leq - B \,\cdot\, \max_i |d_i| \,\cdot\,   \sum_{j\,=\,1}^\ell q_j\,.  \ee
Applying  Lemma~\ref{l(w):ineq} results in the estimate  $\max_i |d_i| \leq  C' (\ell(w) +1 )$ for some constant $C' >0 $ independent of $w$. Thus there exists some constant $C'' >0$, again independent of $w$, such that
\begin{equation}\label{m2:2}
\aligned
-\,\la \Lambda_{\ell+1}, H_2 \ra & \ \ = \ \  r\,\frac{(b, b)}{2}   \ \ = \ \  r\,\frac{(\t{w}b, \t{w}b)}{2} \ \ \leq \ \  - \,C'' \; (\ell(w) +1 ) \; \sum_{j\,=\,1}^\ell r \,q_j \,.
\endaligned \end{equation}
On the other hand, we have from (\ref{rho:h_i}), (\ref{iw2:1}), and (\ref{wb:f}) that
\begin{equation} \label{m2:3}
\la \rho, H_{2, cl} \ra  \ \ = \ \ -\, \la \rho,  r \, \t{w} \,  b \ra \ \  = \ \ - \,r\,\la \rho,  \sum_{j\,=\,1}^{\ell} q_j\, h_j \ra  \ \  = \ \ - \sum_{j\,=\,1}^{\ell} r \,  q_j\,.
\end{equation}
The second inequality of (\ref{iwineq2}) follows from (\ref{m2:2}) and (\ref{m2:3}).


\spoint Finally, we turn to (\ref{iwineq3}) and the last inequality (\ref{iwineq2}) in Theorem~\ref{iwineq}.  Let $H_g:=\ln ( \iwa(g))$, and write
\be{Hg}
 H_g  \ \ = \ \  H_{g, cl}  \  + \  \la \Lambda_{\ell+1}, H_g \ra\, h_{\iota}\,,\ee
where $H_{g, cl} \in
 \mf{h}$ as in (\ref{hhat:decomp}).
If we write $w = \t{w} T_b$ with $\t{w} \in W$ and $b \in Q^{\vee},$ we have
\begin{equation}\label{new3.42}
 \aligned H_3\ \ := \ \  w H_g   & \ \ = \ \  \t{w}\, T_b \, H_{g, cl}  \  + \  \la \Lambda_{\ell+1}, H_g \ra \, h_{\iota}  \\ & \ \ = \ \   \t{w} \,H_{g,cl}  \ + \ \( \la \Lambda_{\ell+1}, H_g \ra + (H_{g,cl} , b ) \) h_{\iota} \,,\endaligned
\end{equation}
where we have used (\ref{affw:actionandcenter}) and the fact that   $w$ fixes $h_\iota$.
This implies \be{H3:cl} (H_3)_{cl} \ \ = \ \  \t{w}\, H_{g, cl}\, , \ee in particular.

Recall that we have assumed that $\eta(s) g \in \hf{S}_t$ with $ t > 0.$  Because of (\ref{At}) we must have that \begin{multline}\label{new3.44} \iwa(   \eta(s) g )^{a_i} \ \  = \ \  e^{ \la a_i, r \,\dd+ H_g \ra }  \ \ > \ \  t  \ \ \ \ \  \text{ for } \ \  i\,=\,1,\, 2,\, \ldots,\, \ell+1\,, \end{multline}
i.e.,  $\la a_i, r \,\dd +H_g\ra > \ln(t)$ for $i=1, \ldots, \ell+1$.
Since (\ref{2.28})-(\ref{a:alpha}) imply that $ \la a_i, h \ra=0$ for $h \in \R h_\iota\oplus \R \dd$ and  $i=1, \ldots, \ell$,
 we furthermore have that \be{Hgcl:abound} \la a_i, r \,\dd +H_g \ra  \ \ = \ \  \la a_i, H_{g, cl}  \ra  \ \ > \ \   \ln(t)  \ \ \ \ \ \text{ for } i=\,1,\, \ldots, \,\ell\,. \ee
 On the other hand, from (\ref{iotaX}) and (\ref{newroot}) we also have that \begin{multline}\label{Hgcl:abound:2}  \la a_{\ell+1},r \,\dd +H_g \ra \ \ = \ \ \la - \alpha_0 + \iota, r \,\dd +H_g  \ra  \ \ = \ \ - \, \la \alpha_0, H_{g, cl} \ra \,+\, r \ \  > \ \  \ln(t)\,  , \end{multline} where we recall that $\alpha_0$ is the highest root for the finite-dimensional Lie algebra $\mf{g},$ and as such is a positive linear combination of the roots $a_i,$ $i=1, \ldots, \ell.$
 Thus, from (\ref{Hgcl:abound}) and (\ref{Hgcl:abound:2}) we conclude that as $\eta(s) g $ varies over the Siegel set $\hf{S}_t$, the quantity $| \la a_i, H_{g, cl} \ra |$ is bounded for $i=1, \ldots, \ell$ by a constant  which depends only on $t, r,$ and the root system $\Delta.$ Since $\t{w}$ varies over a finite set, inequality (\ref{iwineq3}) now follows from (\ref{H3:cl}) and Lemma~\ref{ragh}.

Formula (\ref{new3.42}) also implies that  \be{new3.43} \la \Lambda_{\ell+1} , H_3  \ra  \ \ = \ \  \la \Lambda_{\ell+1}, H_g \ra \  + \  ( H_{g, cl} , b )\,. \ee Using (\ref{iwineq3}), Lemma~\ref{l(w):ineq}, and the Cauchy-Schwartz inequality, we see that there exists a constant $C= C(s,t, \Delta)$ depending on such that $ | (H_{3,cl}, b) | \leq C (\ell(w) + 1).$  Since $\la \Lambda_{\ell+1}, H_g \ra$ depends locally uniformly on
       $H_g$,   the last inequality of (\ref{iwineq2}) follows.

\section{Entire absolute convergence on loop groups} \label{section:convergence}

In this section, we combine our analysis from section~\ref{section:iwasawa} with a decay estimate on cusp forms  to conclude that cuspidal loop Eisenstein series are entire. We begin with some discussion of the  structure of parabolic subgroups of loop groups.


\spoint \label{Mtheta} Let $k$ be any field  and consider the group $\hg_k$ constructed in section~\ref{groups:section}. For any subset $\theta \subset \{ 1, \ldots, \ell+1 \}$  let  $\aw_\th$ denote the subgroup of $\aw$ generated by the reflections $\{w_i|i \in \th\}$, and define the parabolic subgroup $\hP_{\theta, k} \subset \hg_k$ as \be{parabdef} \hP_{\theta,k}  \ \ := \ \  \hB_k \; \aw_\th \; \hB_k\,,
\ee
where each $w\in \wh{W}_{\theta}$ is identified with a representative  in $N(\wh{H}_k)$ as in section \ref{bruhat:section}.  Thus $\wh{P}_{\theta,k}=\wh{G}_k$ when $\theta=\{1,\ldots,\ell+1\}$  (cf.~(\ref{bruhat})). Denote by $\hU_{\th,k}$ the pro-unipotent radical of $\hP_{\th,k}.$ If $\theta \subsetneq \{ 1, \ldots, \ell+1 \}$ the group $\aw_\th$ is finite, and the group $\hu_{\th, k}$ can be written as the intersection of $\wh{U}_k$ with a conjugate by the longest element in $\wh{W}_\theta.$ Recall  the subgroup $\hh_k\subset \wh{G}_k$ defined at the beginning of  section \ref{bruhat:section} as the diagonal operators with respect to the coherent basis ${\mathcal B}$. We now set
\begin{equation}\label{a:theta}
\aligned
\hh_{\theta,k}  \ \   &  =   \ \  \la h \in \hh_k \,|\, h^{a_i}\,=\, 1 \ \   \text{for all} \ i \, \in \,\theta \ra \\ \text{and} \ \ \ \ \ \  H( \theta )_k  \ \ & =  \ \   \la h_i(s) \,|\, \text{ for } i \in \theta \ \text{and}  \ s \in k^* \ra\,.
\endaligned
 \end{equation}
 Analogously to (\ref{hh:dp}), there is an almost direct product decomposition $\hh_k = \hh_{\theta,k} \times H(\th)_k$. Let \be{L:theta} L_{\theta,k}  \ \ = \ \  \la  \chi_\beta(s) \mid \beta \in [\theta], \; s \in k \ra\,, \ee where $[\theta]$ denotes the set of all roots in $\wh{\Delta}$ which can be expressed as linear combinations of elements of $\theta.$ We  now define
 \be{M=AL} M_{\th,k} \ \ = \ \ L_{\th,k} \,\hh_{\th,k} \, , \, \ee
 its semi-simple quotient \be{new4.5} L'_{\th,k}  \ \ = \ \  M_{\th,k} / Z(M_{\th,k})  \ \ = \ \  L_{\th,k} / ( Z(M_{\th,k}) \cap L_{\th,k})\,, \ee and the natural projection \be{new4.6} \pi_{L'_{\th,k}}: M_{\th,k} \ \rr \  L'_{\th,k} \,.\ee With the notation above, $\hP_{\th,k}$ decomposes into the semidirect product  \be{P=MU} \hP_{\th,k} \ \ = \ \  \hu_{\th,k} \rtimes M_{\theta,k}  \ee
 (see \cite[Theorem 6.1]{ga:lg2}).
%
%

\spoint We now suppose that   $k= \R$ and continue  our  convention of dropping the subscript $k$ when referring to   real groups. Recall   the subgroup $\ha$ of $\hh$  defined in section~\ref{groups:section}. Setting    $\ha_{\th} = \hh_{\th} \cap \ha$ and $A(\th) = H(\th) \cap \ha$,
it has the direct product decomposition
\be{ha:th}
 \ha \ \  = \ \  \ha_\th  \, \times \, A(\th)
\ee
as a consequence of the fact the Cartan submatrix corresponding to $\theta$ is positive definite \cite[Lemma~4.4]{kac}.

The decomposition (\ref{M=AL})  is not a direct product, but can be refined to one as follows. Define \be{Mth:1} M_\th^1  \ \ = \ \  \bigcap_{\chi \,\in \, X(M_\th) } \ker(\chi^2)\, ,\ee  where $X(M_\th)$ denotes the set of real algebraic characters of $M_\th.$  (Note that the group  $M_\th^1$  was denoted  $\t{L}_\th$ in \cite{ga:zuck}.)   Then the direct product decomposition \be{mth:dp} M_\th  \ \ = \ \  M_\th^1  \, \times  \, \ha_\th\,.  \ee
holds.

 The group $\wh{G}$ has an Iwasawa decomposition with respect to the parabolic $\wh{P}_\theta$, \be{iwasawa:pr} \hg  \ \ = \ \  \hP_{\th} \, \hk   \ \ = \ \   \hu_\th \, M_\th \, \hk\, . \ee In contrast to  (\ref{iwasawa}), this Iwasawa decomposition is typically not unique since $K_\th:= \hk \cap M_\th$ may be nontrivial; in general  only the  map \be{G:Mth1} \iwam: \hg  \ \ \rr \ \  M_\th/ K_\th
\ee
is well-defined.  Noting that $K_\th $ is a compact subgroup of the finite-dimensional group $M_{\th},$ one in fact has that $K_\th \subset M^1_\th$.  This and the direct product (\ref{mth:dp}) allow us to define the map \be{iw:ahat_th}  \iw_{\ha_\th}: \hg \ \  \rr \ \  M_\th/ M^1_\th  \ \ \cong \ \  \ha_\th, \ee
which factors through $\iw_{M_\theta}$.

\spoint \label{Lprime:section} The  finite-dimensional, real semi-simple Lie group $L'_\th$  admits an Iwasawa decomposition
 \be{iw:L} L'_\th \ \  = \ \  U'_\th \, A'(\th) \, K'_\th \ee
of its own, where  $U'_\th \subset L'_\th$ is a unipotent subgroup,    $A'(\th) = \pi_{L'_\th}( A(\th) )$, and \be{K'th} K'_\th \ \  = \ \ \pi_{L'_\th}(K_\th) \ \ = \ \   \pi_{L'_\th}(\hk \cap M_\th) \,. \ee The projection from $L'_{\th}$ onto $A'(\th)$ will be denoted by \be{iw:L'} \iw_{A'(\th)}: L'_{\th} \rr A'(\th). \ee  The factors $\wh{A}_\theta$ and $A(\theta)$ have trivial intersection in  (\ref{ha:th}) and  the factors in (\ref{mth:dp}) commute.  Therefore the intersection  $Z(M_\theta)\cap A(\theta)=Z(M_\theta^1)\cap \wh{A}_\theta\cap A(\theta)$ is also trivial, and hence
   the map $\pi_{L'_\theta}$ induces an isomorphism
   \be{a':a}
     \pi_{L'_\theta}\,:\,A(\th) \ \  \cong \ \  A'(\th) \,.
      \ee
      Define  the map $\iw_{A(\th)}$ as the composition of the Iwasawa $\ha$-projection of $\hg$  defined after  (\ref{iwasawa}) together with the projection onto the second factor in (\ref{ha:th}),  \be{iwA(th)} \iw_{A(\th)}\,: \ \hg \ \  \stackrel{\iw_{\ha}}{\rr} \ \  \ha  \ \ \rr \ \  A(\th)\,.\ee
      Then  the composition
      \be{iw:a'}  \hg  \ \ \stackrel{\iw_{M_{\th}}}{\rr} \ \  M_\th / K_\th  \ \ \stackrel{\pi_{L'_\th}}{\rr}   \ \ L'_{\th} / K'_\th  \ \ \stackrel{\iw_{A'(\th)}}{\rr} \ \  A'(\th)\, \ee
      coincides with $\pi_{L'_\theta}\circ \iw_{A(\th)}$.

\spoint \label{iw:iota}
We shall fix the choice  $\theta=\{ 1, \ldots, \ell \}$ for the remainder of this section, so that  $\aw_{\theta}= W$.
The groups $\wh{A}_\theta$ and $A(\theta)$ are then respectively isomorphic to the groups $\wh{A}_{cen}$ and $\wh{A}_{cl}$ defined just after (\ref{hh}).
We will thus denote  the maps $\iw_{\ha_\th}$  and $\iw_{A(\th)}$ simply by $\iw_{\ha_{cen}}$ and   $\iw_{\ha_{cl}}$, respectively.
 Since $\eta(s)$ normalizes $\wh{U}$
 we may extend $\mcen$ and $\iw_{\ha_{cl}}$ to maps on the slice $\eta(s) \hg$  which we continue to denote by the same names, \be{mcen:eta} \mcen: \eta(s) \hg \ \ \rr \ \ \ha_{cen} \ \ \ \  \text{ and } \ \  \ \ \iw_{\ha_{cl}}: \eta(s) \hg \ \ \rr \ \ \ha_{cl} \, ,\ee
   similarly to  (\ref{iweadef}).  Likewise we may furthermore define the map \be{iwM:eta} \iw_{M_{\th}}: \eta(s) \hg  \ \  \rr  \ \   M_{\th} / K_\theta
   \ \  = \ \  M_{\th} / (\hk \cap M_{\th}) ,\ee
noting from (\ref{etaandchi}) that $\eta(s)$ acts trivially on $M_{\theta}=M_{\{1,\ldots,\ell\}}$.

   Recall that the group $\wh{A}_{cen}$ consists of the elements $h_{\iota}(s)$ for $s>0$.  For a complex number $\nu \in \C$ define $h_{\iota}(s)^\nu=s^\nu$, and  consider the  function
\begin{equation}\label{Phi_nu}\aligned  \eta(s) \hg  & \ \ \rr \ \  \C^* \\  \eta(s) \, g & \ \ \mapsto \ \  \mcen( \eta(s) \, g ) ^{\nu}\,. \endaligned\end{equation}
 Let $\hp$ be the parabolic $\hp_{\th}=\hp_{\{ 1, \ldots, \ell \}}$.
  The following convergence theorem for Eisenstein series has been proven by the first named author.

\begin{nthm}[Theorem 3.2 in \cite{ga:zuck}]  \label{gar:main}  The Eisenstein series \be{deg:es} \sum_{\gamma \; \in \; ( \hgam \,  \cap \,  \hp ) \backslash  \hgam} \mcen(\gamma \eta(s) \, g)^\nu \ee converges absolutely for $\Re{\nu}>2h^\vee$,   where $h^{\vee}=\langle \rho,h_{\alpha_0}\rangle + 1$ is the dual Coxeter number of $G$. Moreover, the convergence is uniform when
   $g$ is  constrained to a subset of the form $\hu_{\mc{D}} \; \ha_{cpt} \; \hk$, where $\ha_{cpt} \subset \ha$ is  compact  and $\hu_{\mc{D}}$ is as in (\ref{loopsiegel}). \end{nthm}

\noindent   Note that  \cite{ga:zuck}  states absolute convergence for the left half plane $\re(\nu)< -2 h^{\vee}$,  since the order of the Iwasawa decomposition there is reversed.

\spoint Our aim is to prove the convergence of the cuspidal analogs of Theorem~\ref{iw:iota}   for all $ \nu \in \C$.  These were defined in \cite{ga:zuck}, where they were shown to converge in a right half plane.
For any right $K_\th$-invariant function $f: M_{\theta} \rr \C$  the assignment  \be{f:K}
\qquad\qquad\qquad
\eta(s) g  \ \ \mapsto  \ \ f( \iw_{M_\th}(\eta(s) g  ) )  \ \ \ \ \ \ \ \ \ \  \qquad\qquad
\text{(cf.~(\ref{iwM:eta}))}
\ee
 is a well-defined function from $\eta(s)\wh{G}$ to $ \C.$
  Recall that the finite-dimensional semisimple group $L'_{\th}$ was defined in (\ref{new4.5}) together with the projection map $\pi_{L'_\th}: M_\th  \rr   L'_\th$.  Set \be{gamma:th} \Gamma_\th \ \ := \ \  \wh{\Gamma} \, \cap \, M_\th\ee
   and $\Gamma'_\th = \pi_{L'_\th} (\Gamma_\th).$ Recall that  $K'_{\th}=\pi_{L'_\th}(K_\theta)$ from (\ref{K'th}) and let $\phi$ be a $K'_\th$-invariant cusp form on $\Gamma'_\th \setminus L'_\th.$ Set
  \be{vphi} \Phi\,= \,  \phi \circ  \pi_{L'_\th}: \Gamma_\th \setminus M_\th  \ \  \rr \ \  \C\,, \ee which is right $K_\th$-invariant. The \emph{cuspidal loop Eisenstein series} from \cite{ga:zuck} is defined as  the  sum \be{cusp:es} E_{\vphi, \nu}( \eta(s) g ) \ \  := \ \   \sum_{\gamma \, \in\, (\hgam \cap \hp) \backslash \hgam }\mcen(\gamma \,\eta(s)\, g)^{\nu}\, \Phi(\iw_{M_\th}(\gamma\, \eta(s)\, g )) \,. \ee We can now state the main result of this paper

\begin{nthm} \label{main} The cuspidal loop Eisenstein series, $ E_{\vphi, \nu}(\eta(s) g )$ converges absolutely for any $\nu \in \C.$ Moreover, the convergence is uniform when
$g$ is  constrained to a subset of the form $\hu_{\mc{D}} \;  \ha_{cpt} \; \hk,$ where $\ha_{cpt} \subset \ha$ is   compact   and $\hu_{\mc{D}}$ is as in (\ref{loopsiegel}).  \end{nthm}

\spoint The proof of Theorem~\ref{main} is based on two main ingredients, namely the inequalities proved in Theorem~\ref{iwineq}  and also the following decay estimate.  To state it, we keep the notation as in $\S \ref{finrtsystems}-  \ref{sd2}.$ Let $\phi:\Gamma \setminus G\rightarrow\C$ be a $K$-finite cusp form. The following result can be deduced from  Theorem~\ref{thm:uniform} in the appendix.  Stronger results are known for $K$-fixed cusp forms due to Bernstein and Kr\"otz-Opdam \cite{Kr:opdam}, and those are in fact sufficient for the applications here since we only consider Eisenstein series induced from $K$-fixed cusp forms.  Their techniques extend to the $K$-finite setting, but have not been published.  As  mentioned in the introduction, we anticipate the $K$-finite statement will be useful for proving the convergence of $K$-finite loop Eisenstein series once a definition has been given, and so we give a complete proof of the following result in the appendix (using a different argument).

\begin{nthm} \label{decay:body} Let $\phi \in \Gamma \setminus G$ be a $K$-finite cusp form. Then there exists a constant $C > 0$ which depends only on $G$ and  $\phi$ such that for every natural number $N \geq 1,$ we have \be{cor:ufrho}  \phi(g )  \ \ \leq \ \  (CN)^{CN} \ \iw_A(g)^{-N \rho}\,. \ee \end{nthm}

\noindent This is exactly the statement of (\ref{thm:uf}) if $g$ is in some Siegel set $\mf{S}_t$. In general, if $g \in G$ we can find $\gamma \in \Gamma$ such that $\gamma g \in \mf{S}_t.$ By  Remark~\ref{remark231}  we have that \be{mineq} \iw_A(\gamma g )^{\rho}  \ \ \geq \ \ \iw_A(g)^{\rho}\,, \ee
and so the asserted estimate for $\phi(g)=\phi(\g g)$ reduces to the estimate on $\mf{S}_t.$

\tpoint{Remark}\label{remark731} The explicit $N$-dependence in the estimate (\ref{cor:ufrho}) is needed in showing the boundedness of (\ref{est:-1}) below.  Indeed, since $N$ there depends on the Bruhat cell,  the classical rapid decay statements (in which $(CN)^{CN}$ is replaced by some constant depending on $N$) are insufficient.


\spoint\label{sec:4.8.1}
This subsection contains the proof of Theorem~\ref{main}.  Since  cusp forms are bounded,  the convergence for $\re \nu > 2 h^{\vee}$ follows from that of (\ref{deg:es}) (as observed by Garland in \cite{ga:zuck}).
 We shall prove the theorem for   $\re \nu \leq 2 h^{\vee}$  by leveraging the decay of the cusp form to dominate the series (\ref{cusp:es}) by a convergent series of the form (\ref{deg:es}), but with $\nu$ replaced by some $\nu_0 \gg \re\nu.$

\tpoint{ Step 1:~Main Analytic Estimate.} For $w \in W^\th$ let $\hgam(w) = \hgam \cap \hP w \hP$, which is independent of the representative in $\mathbf N$  taken for $w$.
Each set $\wh{\G}(w)$ is invariant under $\wh{\G}\cap \wh{P}$.
 Using the Bruhat decomposition, the Eisenstein series can then be written as \begin{multline}\label{eis:w:r}  E_{\phi, \nu}( \eta(s) \, g) \ \  = \ \  \sum_{w \, \in \,  W^\th} \sum_{ \gamma  \, \in \,  (\hgam \cap \hP) \backslash \hgam(w)  } \mcen(\gamma \, \eta(s) \, g )^{\nu} \  \Phi(\iw_{M_\th}(\gamma \, \eta(s) \, g ))\,.
  \end{multline}
Recall that we have assumed $\Re{\nu} \le 2h^\vee$.  Choose $\nu_0 \in \R$ with $ \nu_0   > 2 h^\vee \ge  \re \nu $, so that
 the  series  (\ref{deg:es}) converges when $\nu$ is replaced by $\nu_0.$ Consider the inner sum on the right hand side of (\ref{eis:w:r}) for a fixed element $w \in W^{\theta},$
\be{es:w}
\sum_{\gamma \, \in \,  (\hgam \cap \hP) \backslash \hgam(w) } \mcen(\gamma \, \eta(s) \, g)^{\nu} \ \Phi(\iw_{M_{\th}}(\gamma \, \eta(s) \, g ) )    \, .
\ee
In the notation of Corollary~\ref{sec:3.4}, let
\begin{multline}\label{y}
 y  \ \ =  \ \ \mcen(\gamma \, \eta(s) \, g)  \ \ \in \ \  \R_{> 0}   \ \    \  \text{ and } \  \ \    x  \ \ = \ \  a^\rho  \ \ = \ \  \iw_{\wh{A}_{cl}} ( \gamma \, \eta(s) \, g )^{\rho}  \ \ \in \ \  \R_{ > 0}\,.
  \end{multline}
By  the definition of $\Phi$ given in (\ref{vphi}),
$
\Phi(\iw_{M_{\th}}(\gamma  \eta(s)   g))  =  \phi( \pi_{L'_{\th}}(\iw_{M_{\th}}(\gamma  \eta(s)  g) ) )$.
Since $\ha_{cl}=A(\th),$ we may conclude from what we have noted after (\ref{iwA(th)})
that
\be{A':Acl} \iw_{A'(\th)}\( \pi_{L'_{\th}} ( \iw_{M_\th}(\gamma \eta(s) \, g)) \)  \ \ = \ \  \pi_{L'_\th} ( \iw_{\ha_{cl}}(\gamma \, \eta(s)\, g) )\,.\ee
Applying
 Theorem~\ref{decay:body}
  we conclude that there exists a constant $C_1 >0$ such that  \be{estimate:cusp}\Phi(\iw_{M_{\th}}(\gamma \,\eta(s) \, g))  \ \ = \ \  \vphi\( \pi_{L'_{\th}} ( \iw_{M_\th}(\gamma \eta(s) \, g))\) &\leq&  (C_1 N)^{C_1 N} x^{-N}  \ee
for any $N\in \Z_{>0}$.

%

\tpoint{Step 2:~Comparing the central contribution.}  From Corollary \ref{cor:iwineq:2} there exist constants $C_2, D >0$ independent of $\gamma \in \hgam \cap \hb w \hb$, but depending locally uniformly on $g$, such that
\be{y:inequ:2}
 x  \ \ \geq  \ \ D \, y^{-(C_2 (\ell(w)+1))^{-1}}\,.
 \ee Choose  a positive real $d$ such that $dC_2  \in \zee_{>0}$ and
\be{d}  d  \ \ > \ \    \nu_0  \ - \  \re \nu   \ \ > \ \  0\,, \ee and let $N=   d  C_2  (\ell(w) + 1).$ From (\ref{y:inequ:2}) we obtain
\begin{multline}\label{est2} x^{-N}\, y^{\re\!\nu  - \nu_0}  \ \ \leq  \ \  D^{-N} \, y^{ \frac{N}{C_2 (\ell(w)+1)} }  y^{\re\!\nu - \nu_0} \ \ = \ \   D^{-N} \, y^{d} \, y^{\re\!\nu - \nu_0}\,. \end{multline} Thus
\begin{equation} \label{estimate:term}
\aligned
\mcen(\gamma \, \eta(s) \, g)^{\re\!{\nu} }\ \Phi(\iw_{M_\th}(\gamma \, \eta(s) \, g )   )    \ \ &  = \ \   y^{\re\! \nu} \  \Phi(\iw_{M_\th}(\gamma \, \eta(s) \, g )) \\  & \leq \ \    (C_1N)^{C_1N}  \ y^{\re\!\nu}  \ x^{-N} \\ & \leq  \ \   (C_1N)^{C_1N}  \ y^{ \nu_0} \  ( x^{-N} y^{\re \! \nu  - \nu_0 })\, \\
 & \leq  \ \ (C_1N)^{C_1N} \,y^{ \nu_0 }\, D^{-N} \, y^a\,, \endaligned \end{equation}
where  $a=d+\re\nu-\nu_0$ is positive by  (\ref{d}).

On the other hand,
\begin{multline}\label{lny}
  \ln y  \ \ = \ \  \la \Lambda_{\ell+1} , \ln(\iw_{\eta(s)\ha} (\gamma \, \eta(s) \, g)) \ra  \ \ = \ \  \la \Lambda_{\ell+1}, H_1 \ra \ +  \ \la \Lambda_{\ell+1} , H_2 \ra  \ + \  \la \Lambda_{\ell+1} , H_3  \ra
   \end{multline}
   in the notation of Theorem~\ref{iwineq}.  The second inequality in (\ref{m:bound}) states that $ \la \Lambda_{\ell+1}, H_1 \ra   \leq   0$. From (\ref{iw2:1}) and Lemma~\ref{l(w):ineq}, the term $\la \Lambda_{\ell+1} ,H_2 \ra$ is bounded above by a quadratic polynomial in $\ell(w)$ with   strictly negative quadratic term. Furthermore, for fixed $\eta(s) g \in \hf{S}_t,$ it follows from (\ref{new3.43}) that $| \la \Lambda_{\ell+1}, H_3  \ra | $ is bounded above by a  linear polynomial in $\ell(w)$ with coefficients that depend locally uniformly in $g$. Hence, in total $y$ is bounded above by an expression of the form $e^{p(\ell(w))}$ where $p(\cdot)$ is a quadratic polynomial with strictly negative quadratic term.  By our choice of $N$
the term $D^{-N} \, (C_1N)^{C_1N}$ is bounded above by a constant times $e^{c' \ell(w) \ln \ell(w)}$ for some positive constant $c'$. Hence, the expression \be{est:-1} (C_1 N)^{C_1N} \,  D^{-N} \, y^a  \ee is bounded as $\ell(w) \rr \infty,$ and so we can   bound  (\ref{cusp:es}) by a sum of the form
\begin{multline}\label{est:0}
 \sum_{w \, \in \,  W^\th} \sum_{ \gamma  \, \in \,  (\hgam \cap \hP) \backslash \hgam(w)  }
  \mcen(\gamma \eta(s) \, g)^{\nu_0} \ \  = \ \
 \sum_{\gamma \,\in\,( \hgam \, \cap \, \hp)  \setminus \hgam  } \mcen(\gamma \eta(s) \, g)^{\nu_0}  \,. \end{multline}
  The convergence of (\ref{cusp:es}) then follows from the fact that  $\nu_0>2h^\vee$ is in the range of convergence of (\ref{deg:es}).


\spoint \label{funfield} Instead of working over $\Q,$ we can also study cuspidal loop Eisenstein series over a function field $F$ of a smooth projective curve $X$ over a finite field. It was in this setting that Braverman and Kazhdan originally noticed the entirety of cuspidal Eisenstein series \cite{bk:ad}. In fact, they observed something stronger which is peculiar to the function field setting. Namely, for every fixed element in the appropriate symmetric space, the Eisenstein series is   a finite sum. Their argument is geometric in nature and relies on estimating the number of points in certain moduli spaces which arise in a geometric construction of Eisenstein series.

Alternatively, one can also easily adapt the arguments in this paper to the function field setting to reprove the result of Braverman and Kazhdan. We briefly indicate the argument here. As in the proof of Theorem~\ref{main}, we first break up the Eisenstein series into a sum over Bruhat cells indexed by $w \in W^{\theta}.$ On each such cell, it is easy to see using the arguments in \cite[Lemma~2.5]{ga:ms2} that there can be at most finitely many elements whose classical Iwasawa component lies in the support of any compactly supported function on $(Z(M_{\th,\ad_F}) M_{\th, F})\backslash M_{\th, \ad_F}/(K_{\ad_F}\cap M_{\th, \ad_F})$. It is a result of Harder (see \cite[Lemma~I.2.7]{mw}) that cusp forms do indeed have  compact support on their fundamental domain.  Hence, only finitely many terms contribute to the cuspidal loop Eisenstein series on each cell. Moreover, one can use an analogue of our main Iwasawa inequalities Theorem~\ref{iwineq} to show that cells corresponding to $w$ with $\ell(w)$ sufficiently large cannot contribute \emph{any} non-zero elements to the Eisenstein series. Indeed, as in the proof of Theorem~\ref{main}, one may again show that as $\ell(w) \rr \infty$ the central term of the Iwasawa component is bounded above by $e^{p(\ell(w))}$ where $p(\cdot)$ is a quadratic polynomial with strictly negative quadratic term. On the other hand, the analogue of (\ref{y:inequ:2}) dictates that as $\ell(w) \rr \infty$ the central piece grows smaller and the classical piece must grow large and therefore eventually lie outside the support of the cusp form.

\appendix

\section{$K$-finite cusp forms decay exponentially} \label{appendix}

\subsection{Introduction and statement of main result}

In this appendix we give a proof of  Theorem~\ref{decay:body}.
This  rapid decay estimate with uniform constants will be deduced from  an exponential-type decay statement.  Historically rapid decay estimates (without uniformity) were established in order to circumvent   the apparent difficulty of proving the exponential decay estimates familiar in the classical theory of automorphic forms for $SL(2,\R)$.  As our argument here indicates, one gets qualitatively similar estimates for  $K$-finite cusp forms on general groups by injecting a small amount of one-variable hard analysis.  Stronger exponential decay results were proven by Bernstein using the more sophisticated technique of holomorphic continuation of representations (for which we refer to Kr\"otz-Opdam \cite{Kr:opdam} for $K$-fixed cusp forms, and unpublished notes of theirs for $K$-finite forms).
 Although we state and prove the results here in the setting of Chevalley groups used in the body of the paper, the results and the methods are applicable to general groups.

  The application to Eisenstein series in the body of this paper uses only $K$-fixed forms, since (as noted in the introduction) at present it is difficult to formulate a definition of Eisenstein series induced from general $K$-finite forms without further developments in the representation theory of affine loop groups.  However, this issue seems independent of the analysis we used to demonstrate convergence, and so the bounds we prove here should ultimately play the same role in showing the everywhere absolute convergence (and hence entirety) of those series as well.  Since this estimate does not appear in the literature we have chosen to include this appendix.

 Let $G$ be a Chevalley group over $\R$ and $\Gamma$ an arithmetic subgroup with respect to the Chevalley structure. Other notations used here will have the same meaning as in the main body of this paper (see \S\ref{section:notation}A,  in particular (\ref{Atdef})).

\begin{nthm}\label{mainthm}   Let $\phi$ be a $K$-finite cusp form on $\G\backslash G$.   For any fixed $t >0$  there exist constants $c,d>0$ such that $\phi$ satisfies the estimate
\begin{equation}\label{mainestimate}
    \phi(ak) \ \ \le \ \ \exp\(-c |\max_{\a\,\in\,\Delta_+}a^\a|^d\)\ \ \ \ \  \text{for all} \ \ a\,\in\,A_t\,.
\end{equation}
\end{nthm}

\noindent
The constant $c>0$ is certainly necessary,
as can already been seen in the context of  classical modular forms on the complex upper half plane.   Indeed, a cusp form invariant under parabolic transformations $z\mapsto z+w$ is bounded by a multiple of $e^{-2\pi w^{-1} \text{Im}(z)}$.  Bernstein's holomorphic continuation method mentioned above shows that $d\ge 1$, in consonance with classical modular forms.
 Up to logarithmic factors, we are not aware of any examples where $d$ is not 1, though we make some comments in \S\ref{stringtheory} about where they might occur.  Even with $d\neq 1$ the estimate (\ref{mainestimate}) could be regarded as exponential decay from the vantage point of the Lie algebra   $\frak h$ of $A$, where
 dilation by a constant multiple has the effect of scaling $d$ (hence the title of this appendix).

The rest of the appendix is arranged as follows:~in $\S\ref{specialfunction}$ we first study  the decay of a certain  function on $G$  whose properties will be crucial later in the argument.
 Combining this estimate with a slight enhancement of a result of Harish-Chandra (proposition $\ref{HCprop}$), we then retrace our way in $\S \ref{uniform}$ through the usual proof of rapid decay to obtain the asserted rapid decay with uniform constants.  We then    deduce  exponential decay from this estimate in $\S \ref{expdecay}$. Finally, in \S\ref{counterexample}-\ref{stringtheory}  we aim to place this result in a broader context, by giving an example which demonstrates the necessity of the $K$-finite condition and explaining the importance of estimates of the form (\ref{mainestimate}) in string theory.

\subsection{Derivative estimates on a particular function} \label{specialfunction}

\aspoint We begin by estimating the norm of a certain function on the real line. Let $\sigma\in C^\infty_c(\R)$ be the ``bump'' function
\begin{equation}\label{f1choice}
    \sigma(x) \ \ := \ \ \left\{
                        \begin{array}{ll}
                          e^{-1/(1-x^2)}, & |x|<1\,, \\
                          0\,, & \hbox{otherwise\,,}
                        \end{array}
                      \right.
\end{equation}
which is supported in the interval $[-1,1]$.

\begin{nlem} \label{onevar} The function $\sigma$ satisfies the $L^1$-norm estimate
\be{onevar:eq}
\left\| \smallf{d^N \sigma}{dx^N} \right\|_1 \ \   \ll \ \ (2N)^{2N}\ee
for positive integers $N$.
 \end{nlem}
\begin{proof}
A standard application of the saddle point method  (e.g., the note
\url{http://math.mit.edu/~stevenj/bump-saddle.pdf}  -- which uses a different normalization of Fourier transform)  shows that its Fourier transform satisfies the bound $\widehat{ \sigma}(r)\ll r^{-3/4}e^{-\sqrt{2\pi r}}$.  In particular
\begin{equation}\label{f1moment}
\aligned
\left\|\widehat{ \sigma}(r)r^{N}\right\|_2^2 \ \ = \ \ \int_\R\widehat{ \sigma}(r)^2r^{2N}dr \ \ & \ll \ \  \int_0^\infty e^{-\sqrt{8\pi r}}r^{2N-3/2}dr
\\ \ll \ \ (2\pi)^{-2N}\,\G(4N-1) \ \ & \le  \ \ (2\pi)^{-2N}\,(4N)! \ \ \ll \ \ (2\pi)^{-2N}\,(2N)^{4N}
\endaligned
\end{equation}
and so by Cauchy-Schwartz,   the compact support of $\sigma$, and Parseval's Theorem we have
\begin{equation}\label{f1moment2}
  \left\|\smallf{d^N \sigma}{dx^N}\right\|_1 \ \ \ll \ \    \left\|\smallf{d^N \sigma}{dx^N}\right\|_2 \ \ = \ \   (2 \pi)^N\,\|\widehat{  \sigma}(r)r^N\|_2 \ \ \ll \ \ (2N)^{2N}\,.
\end{equation}
\end{proof}

\noindent
For further reference we note that
\begin{equation}\label{f1moment3}
    \left\| \smallf{d^N}{dx^N}\,\sigma(c\i x) \right\|_1 \ \ \ll \ \ (2Nc^{-1})^{2Nc^{-1}}
\end{equation}
for $c<1$.
Incidentally, one might wonder whether  replacing  $\sigma$  in (\ref{f1choice}) with a different choice of function might lead to significantly better estimates later on.  Other than improving the constant ``2'' in (\ref{onevar:eq}) this is impossible \cite{sunyer},  consistent with the fact one cannot improve upon the   known exponential decay of classical holomorphic cusp forms.

\aspoint We next use $\sigma$ construct a family of bump functions on our group  $G$ with Lie algebra ${\frak g}={\frak n}_{-}\oplus {\frak h}\oplus {\frak n}_{+}$.  Let $\{Y_1,Y_2,\ldots,Y_k\}$, $\{H_1,\ldots,H_r\}$, and  $\{X_1,X_2,\ldots,X_k\}$ denote bases of these three summands, respectively.  Combined they give a basis
\begin{equation}\label{Bbasis}
{\mathcal B} \ \ = \ \ \{Y_1,\ldots,Y_k,H_1,\ldots,H_r,X_1,\ldots,X_k\}
\end{equation}
 of the $d=2k+r$ dimensional Lie algebra $\frak g$.  Let
\begin{equation}\label{tmaps}
\aligned
    n_{-}(t_1,\ldots,t_k) \ \ & := \ \ e^{t_1 Y_1}e^{t_2 Y_2}\cdots e^{t_k Y_k}\,,   \\
    a(t_{k+1},\ldots,t_{k+r}) \ \ & := \ \ \ e^{t_{k+1}H_1}\cdots
    e^{t_{k+r}H_r} \,, \ \ \text{and} \\
    n_{+}(t_{k+r+1},\ldots,t_d) \ \ & := \ \ e^{t_{k+r+1}X_1}\cdots e^{t_d X_k}\,.
\endaligned
\end{equation}
Then for
    $c>0$ sufficiently small the image of the set $\{|t_i|\le c|1\le i \le d\}$ under the map
\begin{equation}\label{exponentialcoordinatemap}
    (t_1,t_2,\ldots,t_d) \ \mapsto \  n_{-}(t_1,\ldots,t_k)\,a(t_{k+1},\ldots,t_{k+r}) \, n_{+}(t_{k+r+1},\ldots,t_d)
\end{equation}
diffeomorphically
sweeps out a compact neighborhood of the identity in $G$.  Define a smooth function $\psi=\psi_c$ supported in this neighborhood by the formula
\begin{equation}\label{psidef}
     \psi_c(n_{-}\,a\,n_{+}) \ \  = \ \ \prod_{i\,=\,1}^d \sigma(c\i  t_i )\,,
\end{equation}
where $n_{-}an_{+}$ represents the righthand side of (\ref{exponentialcoordinatemap}).

Let $\delta$ denote the right translation operator on functions on $G$.  This action extends to $\frak g$ as the left-invariant differential operators
\begin{equation}\label{rhoX}
    (\delta(X)f)(g) \ \ = \ \ \left.\smallf{d}{dt}\right|_{t=0}f(g e^{tX})\,, \ \ \ X\,\in\,{\frak g}\,,
\end{equation}
and then to  the universal enveloping algebra ${\frak U}$ of $\frak g$  by composition.
\begin{nprop}\label{psiderivbd}
For any degree $N$ element $X\in \frak U$
\begin{equation}\label{psiderivbd1}
    \| \delta(X)\psi_c \|_1 \ \   \ll \ \ (c'N)^{c'N}\,,
\end{equation}
where $c'>0$ is a constant that depends on $c$ and $G$.
\end{nprop}
\begin{proof}
Let us first consider the case where $X\in \frak g$ and the righthand side of (\ref{psidef}) has the more general form
$\prod_{i\le d}\sigma_i(c\i t_i )$, where $\sigma_1,\ldots,\sigma_d \in C^\infty_c(\R)$.  Commuting the derivative in $X$ across the $n_{+}$ factor
\begin{equation}\label{psiderivbd2}
\aligned
    \delta(X)\psi_c(n_{-}\,a\,n_{+}) \ \ & = \ \ \left.\smallf{d}{dt}\right|_{t=0}\psi_c(n_{-}\,a\,n_{+}\,e^{tX}) \\
    & = \ \ \left.\smallf{d}{dt}\right|_{t=0}\psi_c(n_{-}\,a\,e^{t\Ad(n_{+})X}\,n_{+})\,.
\endaligned
\end{equation}
Since $\psi_c(g)$ is supported on elements $g$ which factorize as $g=n_{-}an_{+}$, in which the support of each of the three factors $n_{-}$, $a$, and $n_{+}$ is constrained to a fixed compact set, we observe that
  $\Ad(n_{+})X$ can be written as a linear combination of elements of $\mathcal B$ with bounded coefficients.  Furthermore, since $a\in A$ normalizes ${\frak n}_{-}$ the same compactness argument allows us to write
\begin{equation}\label{psiderivbd3a}
      \delta(X)\psi_c(n_{-}an_{+}) \ \ = \ \ \left.\smallf{d}{dt}\right|_{t=0}\psi_c(n_{-}e^{tY'}ae^{tH'}e^{tX'}n_{+})\,,
\end{equation}
where $Y'\in \frak n_{-}$, $H'\in \frak a$, and $X'\in \frak n_{+}$ depend on $X$, and range over bounded subsets as $X$ varies over the basis $\mathcal B$.  In particular,  the coefficients of $Y'$, $H'$, or $X'$ when expanded in the basis $\mathcal B$ are bounded by a constant $M$ (depending only on $c$ and $G$) as $X$ ranges among the  vectors  ${\mathcal B}$.
 Thus for $X\in {\mathcal B}$
\begin{multline}\label{psiderivbd3b}
     |  \delta(X)\psi_c(n_{-}an_{+}) | \ \ \le \\  Mc^{-1}\,\sum_{i\,=\,1}^d \left| \sigma_1(c\i t_1)\cdots \sigma_{i-1}(c\i t_{i-1})\sigma_i'(c\i t_i)\sigma_{i+1}(c\i t_{i+1})\cdots \sigma_d(c\i t_d)\right| \,.
\end{multline}
Similarly, when $X$ is the tensor product of $N$ basis vectors we have the estimate
\begin{equation}\label{psiderivbd4}
    |  \delta(X)\psi_c(n_{-}an_{+}) | \ \ \le   \ \  M^Nc^{-N}\,\sum_{\srel{(j_1,\ldots,j_d)\,\in\, \Z^d_{\ge 0}}{j_1+\cdots+j_d\,=\,N}}\left| \sigma_1^{(j_1)}(c\i t_1)\,\sigma_2^{(j_2)}(c\i t_2)\cdots \sigma_d^{(j_d)}(c\i t_d)\right|\,,
\end{equation}
as can be seen by iteration.

  In the (compact) range of support of $\psi$, the Haar measure $dg$ in the coordinates (\ref{exponentialcoordinatemap}) is bounded by a multiple of
   $dt_1\cdots dt_d$ and so by the lemma
   \begin{equation}\label{psiderivbd5}
   \aligned
    \|\delta(X)\psi_c\|_1 \ \  & \ll \ \ M^Nc^{-N+d}\sum_{\srel{(j_1,\ldots,j_d)\,\in\, \Z^d_{\ge 0}}{j_1+\cdots+j_d\,=\,N}}\|\sigma_1^{(j_1)}\|_1\cdots \|\sigma_d^{(j_d)}\|_1 \\
    & \ll \ \ M^Nc^{-N+d}\sum_{\srel{(j_1,\ldots,j_d)\,\in\, \Z^d_{\ge 0}}{j_1+\cdots+j_d\,=\,N}}
    (2j_1)^{2j_1}\cdots (2j_d)^{2j_d}\\
    & \le \ \ M^Nc^{-N+d}\sum_{\srel{(j_1,\ldots,j_d)\,\in\, \Z^d_{\ge 0}}{j_1+\cdots+j_d\,=\,N}}
    (2N)^{2j_1}\cdots (2N)^{2j_d} \,.
    \endaligned
   \end{equation}
  The sum above is equal to the multinomial expansion of $(\overbrace{4N^2+\cdots+4N^2}^{N\ \text{times}})^N=(4N^3)^N$, and hence the lemma follows.
\end{proof}

\aspoint Next, we generalize the above estimate to the $K$-twists of the function $\psi$ defined by \begin{equation}\label{fdef}
    f_c(g) \ \ := \ \ \int_{K}\psi_c(k^{-1}gk)\,dk\,.
\end{equation}
For any fixed $X\in {\frak g}$ the elements $\Ad(k^{-1})X$ are bounded as $k$ ranges over the maximal compact subgroup $K$.  In particular if $X \in {\mathcal B}$ then
\begin{equation}\label{rhofbd1}
    (\delta(X)f_c)(g) \ \
    = \ \
    \left.\smallf{d}{dt}\right|_{t=0} \int_K\psi_c(k^{-1}ge^{tX}k)\,dk \ \
     \ll \ \ \sum_{X'\,\in\,{\mathcal B}} \int_K |(\delta(X')\psi_c)(k^{-1}gk)|\,dk\,,
\end{equation}
with an implied constant that depends only on  $K\subset G$.

\begin{nprop} There exists a constant $c'>0$ which depends only on $c$ and $G$ such that
\begin{equation}\label{rhofbd2}
    \|\delta(X)f_c\|_1 \ \ \ll \ \ \sum_{X'\,\in\,{\mathcal B}} \|\delta(X')\psi_c\|_1\ \ \ll \ \
    (c'N)^{c'N}
\end{equation}for all $X \in {\mathcal B}$. \end{nprop}
\begin{proof}
The first inequality follows from (\ref{rhofbd2}) using the bi-invariance of the Haar measure, and the second from proposition~\ref{psiderivbd}.
\end{proof}

\aspoint A result of Harish-Chandra \cite{hc} (see also \cite[p.~22]{bor}) asserts that for any $K$-finite automorphic form $\phi$, there exists a function $f\in C_c^\infty(G)$ such that $\phi$ is  in fact {\em equal} to  the convolution
\begin{equation}\label{convdef}
    (\phi\star f)(g) \ \ := \ \ \int_G \phi(gh\i)\,f(h)\,dh\,.
\end{equation}
The proof furthermore shows that $f$
satisfies the condition  $f(k^{-1}gk)=f(g)$ for all $k\in K$, and
can be taken to be  a linear combination of functions in an approximate identity sequence.  For $\psi_c$ of the form (\ref{psidef}), the $K$-twists defined in (\ref{fdef}) are approximate identity sequences.  This allows us to conclude the following
\begin{nprop}\label{HCprop}
    For any $K$-finite automorphic form $\phi$ there exists a finite number of positive real numbers $c_i>0$ and complex numbers $a_i$ such that
     \begin{equation}\label{refinement}
        \phi \ \ = \ \ \sum_i a_i \,(\phi\star f_{c_i})\,.
     \end{equation}
\end{nprop}

\noindent
For completeness we give the proof, which is completely based on ideas of Borel and Harish-Chandra as in \emph{op. cit}.
\begin{proof}
As $c\rightarrow 0$ the functions $f_c$ in (\ref{fdef}) are smooth, nonnegative, and have support shrinking to zero.   Thus for any $\phi_1\in C_c^\infty(G)$  the convolutions $m_c\i (\phi_1 \star f_c)$ converge to  $\phi_1$ uniformly on compacta, where $m_c=\|f_c\|_1$.  At the same time, changing variables in (\ref{fdef}) demonstrates that  $f_c(k\i g k)=f_c(g)$, and so $\phi\star f_c$ lies in $V$ := the span of all $K$-translates of $\phi$, a finite dimensional space.    (This is the only point in this appendix where the $K$-finiteness property is used.)   The linear span  $W$ of the convolution operators $\{\star f_c|c>0\}$ is a subspace of $\Aut(V)$ that is closed under any of its equivalent topologies, being that both are finite dimensional.  The particular sequence of operators  $\{\star(m_c\i f_c)\}$ converges to the identity operator on $V$ under the topology of uniform convergence on compacta, and hence the identity operator is in $W$; that is, the identity operator on $W$ has the form $\star\sum_i a_i f_{c_i}$, which implies (\ref{refinement}).
\end{proof}

\aspoint \label{final:special} We shall now combine the results of the previous paragraphs.  According to (\ref{convdef}),
\begin{equation}\label{convderiv}
    \delta(X)(\phi\star f) \ \ = \ \ \phi \star \delta(X)f
\end{equation}
for any $X\in \frak U$, as can be seen by changing variables.  Thus the boundedness of cusp forms combined with estimate (\ref{rhofbd2}) shows that
\begin{equation}\label{Linftybd}
    \|\delta(X)(\phi \star f_c)\|_\infty \ \ \le \ \ \|\phi\|_\infty\, \|\delta(X)f_c\|_1 \ \ \le  \ \ (c'N)^{c'N}\,,
\end{equation}
  where $N$ is the degree of $X$ and $c'>0$ is a constant depending on $c$ and $G$.
Proposition~\ref{HCprop} shows that $\phi$ is a finite linear combination of convolutions   $\phi\star f_{c_i}$, and so we   conclude

\begin{nprop} \label{spfcor} Let $\phi$ be a $K$-finite cusp form on $G.$ There exists a positive constant
 $C >0$ depending only on $\phi$   such that for every $X \in \mf{U}$ of degree $N$ \be{Linftybd2}
   \|\delta(X)\phi \|_\infty \ \ \le \ \ (CN)^{CN}\,. \ee  \end{nprop}
%

\subsection{Rapid decay with uniform constants} \label{uniform}

In this section we use  proposition~\ref{final:special}  to give a rapid decay estimate with uniform constants.  The methods here are fairly standard (see \cite[\S2.10]{mw}, for example), though it is unfortunately necessary to reproduce them  in order to demonstrate how the constant factors in (\ref{Linftybd2}) propagate.

\aspoint Recall the conventions of \S \ref{sd1} and \S \ref{sd2},
where we  fixed an Iwasawa decomposition $G= U A K$ and a projection $\iw_A: G \rr A.$
 Furthermore, in (\ref{atolambdef}) we defined a character  $A \rr \C^*$ which sends $a \mapsto a^{\mu}$
  for each linear functional $\mu: \mf{h}=\text{Lie}(A) \rr \C$.

\begin{nthm} \label{thm:uniform} Let $\phi$ be a $K$-finite cusp form on $\Gamma\backslash G.$ Fix a
 Siegel set $\mf{S}_t$ for $G$. Then there exists a constant $C > 0$ which depends only on $\phi$ and the parameter $t$ of the Siegel set $\mf{S}_t$ such that
 \be{thm:uf}  \phi(g )  \ \ \leq \ \  (CN)^{CN}\  \iw_A(g)^{-N \rho}     \ \ \ \  \text{  for all } g \in \mf{S}_t
  \  \, \text{and} \, \ N\ge 1\,. \ee
  \end{nthm}

The remainder of this section will be devoted to the proof of this theorem.

\aspoint Let $\alpha \in \Pi$ be a simple root and $P$ the corresponding
 standard maximal parabolic subgroup, with unipotent radical $U_P$ and
 corresponding Lie algebra $\mf{n}_P.$ Let $\Delta_{+, \alpha}$ denote
 the set of positive roots \be{delta:alpha} \{ \gamma \,\in\,\Delta_+ \ |\  \gamma  = \sum_{\beta \in \Pi} m(\beta) \beta \text{ with } m(\alpha) \neq 0 \}\,,  \ee
 which are precisely the roots $\g$ whose Chevalley basis root vector
  $X_\g$ lies in ${\frak n}_P:=\text{Lie}(U_P)$.
 Choose an ordering of the roots of $\Delta_{+,\alpha}$ such that if
  $\beta, \gamma, \beta+ \gamma \in \Delta$ then $ \beta +
  \gamma > \beta$ and $\beta + \gamma > \gamma$  : \[ \beta_1 \,  > \,  \beta_2 \,  >
   \, \cdots  \, > \,  \beta_p \, . \] For $i=1,\ldots,p$ let $X_i=X_{\b_i}$ be the Chevalley
    basis root vector  corresponding to the root $\beta_i.$
Corresponding to the set $\{ \beta_1, \ldots, \beta_i \}$ for $0 \leq i \leq p,$ we let
 $\mf{n}_i \subset \mf{n}_P$ and $U_i \subset U_P$ denote the corresponding nilpotent
 subalgebra and normal unipotent subgroups of ${\frak n}_P$ and $U_P$, respectively;
 that is, ${\frak n}_i$ is the real span of $\{X_1,\ldots, X_i\}$ and $U_i$ is its exponential. We then have the following ascending chains of Lie algebras and unipotent groups:
\be{nilp:filtration}
\{0\} \   =  \  \mf{n}_0 \subset \mf{n}_1 \subset &\cdots& \subset \mf{n}_p= \mf{n}_P \\
\{0\} \  = \   U_0 \subset U_1 \subset &\cdots& \subset U_p= U_P. \ee
For   $1 \le  i \leq p$  the one parameter subgroups  $\chi_{\b_i}(s)=e^{sX_i}$ induce isomorphisms
$U_{i-1}\backslash U_i\cong \R$;
letting $\Gamma_{\infty} = \Gamma \cap U$, $\Gamma_i = \Gamma \cap U_i$,
 and rescaling the $X_i$ if necessary, they furthermore induce isomorphisms  \be{R/Z} \chi_{\beta_i} \, : \,
    \zee \backslash \R \ \  \cong  \ \ ( U_{i-1} \Gamma_i)\backslash U_i   \ee
that will be used to parametrize Fourier series below.  Since replacing $X_i$ by a positive multiple
will make no difference in the rest of this appendix,
we shall assume this rescaling has been performed.

\aspoint For any function $\Psi\in L^1(\G\backslash G)$ we can define projections
\be{psiidef}
\Psi_i(g) \ \ := \ \  \int_{\Gamma_i \backslash U_i } \Psi(u_i g)\, du_i \ \, , \ \ \ 0 \,\le\, i \,\le\,p\,, \ee  where $du_i$ is the Haar measure on $U_i$ normalized to give the quotient $\Gamma_i\backslash U_i$ measure one.
These Haar measures can be written as products of Haar measures along the one-parameter subgroups given by the characters $\chi_{\beta_j}$ for $j\le i$.   Using (\ref{R/Z}) we can write  \[ \Psi_i(g)  \ \ = \ \  \int_{ ( U_{i-1} \Gamma_i )\backslash U_i } \Psi_{i-1}(u g) \,du \ \ = \ \
\int_{\Z\backslash \R} \Psi_{i-1}(\chi_{\b_i}(s)g)\,ds
   \]
   and hence
\be{phi:fourier}
\aligned \Psi_{i}(g) \, - \,  \Psi_{i-1}(g)  \ \ &  = \ \
 \int_{0}^1\left[\,\Psi_{i-1}(\chi_{\beta_i}(s)g) \ - \ \Psi_{i-1}(\chi_{\beta_i}(0)g)\,\right]\,ds\\
 & = \ \
 \, \int_{0}^1
  \int_0^s \(\smallf{d}{dr}\Psi_{i-1}(\chi_{\beta_i}(r)g)\)
  \,dr
  \,ds\,.
 \endaligned
 \ee
 Thus
 \begin{equation}\label{absolutevaluebd}
    |\Psi_{i}(g)\,-\,\Psi_{i-1}(g)| \ \ \le \ \ \max_{0\, \le \,  r \, \le \,  1} \left|\smallf{d}{dr}\Psi_{i-1}(\chi_{\beta_i}(r)g)
    \right|
 \end{equation}
 is bounded in terms of derivatives of $\Psi_{i-1}$ under {\em left} translation by the one-parameter subgroup generated by $X_i$.

 We next describe these left derivatives in terms of the right action (\ref{rhoX}):
 \begin{equation}\label{lefttoright}
    \smallf{d}{dr}\Psi_{i-1}(\chi_{\beta_i}(r)g) \ \ = \ \
   \(\delta((\Ad{g}^{-1})X_i )\Psi_{i-1}\)(\chi_{\beta_i}(r)g)\,,
 \end{equation}
 which can be seen by passing the Lie algebra derivative across the
 argument $\chi_{\beta_i}(r)g$ of $\Psi_{i-1}$.  Suppose now that
 $g$ has a factorization $g=uak$, in which  $u$ and $k$ range over fixed compact sets;
 this description in particular applies to elements of Siegel sets.  Because of our
 assumption on the ordering $\,>\,$, $\Ad(u^{-1})X_{\beta_i}$ can be written as a
  linear combination of the root vectors $\{X_{\beta_1},\ldots,X_{\beta_i}\}$ with
  bounded coefficients; the precise bound depends on the size of the compact set $u$ ranges over,
  and in particular depends only on $\G$ in the circumstance that $g\in \mf{S}_t$.
   For each $a\in A$, $\Ad(a^{-1})X_{\beta_j}=a^{-\beta_j}X_{\beta_j}$; because of
   (\ref{delta:alpha}) and the definition of $\mf{S}_t$, the factor $a^{-\beta_j}\ll a^{-\a}$
   with an implied constant depending only on $t$ and $G$.  Finally, $\Ad(k^{-1})$ maps any fixed element
    of $\frak g$ to a bounded set.  Putting together these statements about $\Ad(u^{-1})$, $\Ad(a^{-1})$,
    and $\Ad(k^{-1})$, we see that $\Ad(g^{-1})X_i=\Ad(k^{-1})\Ad(a^{-1})\Ad(u^{-1})X_i$ is a linear combination of elements of
    (\ref{Bbasis}) with coefficients bounded by $a^{-\a}$ times a constant that depends only on the Siegel
     set $\mf{S}_t$.
  Combining (\ref{absolutevaluebd}), (\ref{lefttoright}), the sup-norm inequality $\|\d(X)\Psi_i\|_\infty\le\|\d(X)\Psi\|_\infty$, and this expansion of
   $\Ad(g^{-1})X_i$, we see that
  \begin{equation}\label{pretelescope}
    \left| \Psi_{i-1}(g) \,-\,\Psi_i(g)\right| \ \ \ll \ \ a^{-\a} \,
    \max_{X\,\in\,{\mathcal B}}\|\d(X)\Psi_i\|_\infty \ \ \le \ \ a^{-\a} \,
    \max_{X\,\in\,{\mathcal B}}\|\d(X)\Psi\|_\infty  \,,
  \end{equation}
  with an implied constant that depends only on ${\frak S}_t$.

 Specialize $\Psi$ to be the cusp form $\phi$, so that $\Psi_p=\phi_p=0$. By applying the estimate (\ref{pretelescope})
 to the telescoping sum
\begin{equation}\label{telescope}
  \phi(g)  \ \ = \ \   \sum_{i=1}^p \(\phi_{i-1}(g) \, -\, \phi_i (g)\)\,,
\end{equation}
we conclude that $\phi(g)$ is bounded by the product of  $a^{-\a}$,
the sup norms $\|\d(X)\phi\|_{\infty}$ for any $X$ in (\ref{Bbasis}), and a constant that depends only on  $\mf{S}_t$.
Theorem~\ref{thm:uniform} then follows from iterating this procedure for various simple roots $\a$
 and applying the estimate in proposition~\ref{final:special}.%

Recall that  Theorem~\ref{decay:body} is deduced from Theorem~\ref{thm:uniform} immediately after its
 statement on page~\pageref{mineq}.

\subsection{Exponential decay} \label{expdecay}

\begin{nlem} Let $F:\R_{>0}\rightarrow \C$ be a bounded function for which there exists   constants  $c>0$ and $C>0$ such that
\begin{equation}\label{exponentiallemma1}
    F(y) \ \ \le \ \ c\,y^{-N}\,(CN)^{CN} \ \ \ \text{for all} \ \, N\,\ge\,1\,.
\end{equation}
Then there exists   constants $d>0$ and $D>0$ such that
\begin{equation}\label{exponentiallemma2}
    F(y) \ \ \le \ \ D  \, \exp(-  y^{d})\,.
\end{equation}
\end{nlem}

The converse is of course elementary.
Theorem~\ref{thm:uniform} implies Theorem~\ref{mainthm}
as a consequence of this lemma, and also vice-versa (though with different constants).

\begin{proof}
By adjusting the constants $d$ and $D$ it suffices to shower the weaker bound $F(y)\le D y e^{-y^d}$ for $y>1$.
If  $y>1$   the assumption (\ref{exponentiallemma1}) implies
\begin{equation}\label{exponentiallemma4}
    y^{dn-1}\,F(y) \ \ \le \ \   y^{\lfloor dn\rfloor}\,F(y) \ \ \le \ \
     c (C\lfloor dn\rfloor)^{C\lfloor dn\rfloor} \ \ \le \ \ c(Cdn)^{Cdn}
\end{equation}
for any $d>0$ and $n\in \Z_{>0}$.
For $d$ sufficiently small the righthand side is less than a constant times
$(n!)/n^2$, so summing over $n$ gives the boundedness of
\begin{equation}\label{exponentiallemma3}
   y^{-1}\,\exp( y^{d})\,F(y) \ \ = \ \ \sum_{n\,=\,0}^\infty \smallf{1}{n!}\,y^{dn-1}\,F(y)
\end{equation}
in the range $y>1$, as was to be shown.
\end{proof}

\subsection{Some rapidly, but not exponentially, decaying smooth cusp forms}  \label{counterexample}

In this section we give an example of a non-$K$-finite yet smooth
 cusp form which does not satisfy the exponential-type decay statement
present in Theorem~\ref{mainthm}.  It is a theorem of \cite{rapiddecay} that all smooth cusp forms have
   rapid decay, but the proof does not
   give the uniformity of constants  present in the stronger statement
    (\ref{thm:uf}) from Theorem~\ref{thm:uniform}.  The example here shows that not only is this uniformity
     impossible, but furthermore that its failure is very common:~this is because we show that the decay
     is governed by an arbitrarily
    chosen Schwartz function (relatively few of which decay faster than an exponential).

We give an example for the group $SL(2,\R)$, though the phenomenon is certainly much more general.
Let $\pi$ be any cuspidal automorphic representation for $SL(2,\Z)\backslash SL(2,\R)$.
 We will demonstrate it has non-$K$-finite smooth vectors which violate (\ref{mainestimate}).
   To do this it is convenient to use the notion of {\em automorphic distribution} as in  \cite{microsoft}.
    The representation $\pi$ has an automorphic distribution $\tau$ which embeds smooth vectors $v$ for some principal series representation to automorphic forms $\phi$:
\begin{equation}\label{tauembed}
    \phi(n_xa_y) \ \ = \ \ \langle \tau,\pi(n_xa_y)v\rangle\,, \  \ \text{where} \ n_x=\ttwo1x01 \ \text{and} \ a_y=\ttwo{y^{1/2}}{0}{0}{y^{-1/2}}.
\end{equation}
 Such an embedding exists even when $\pi$ is itself not a principal series representation, e.g., a discrete series representation.
In the line model, the Schwartz space  is an embedded subspace of smooth vectors. If $v$ corresponds to a Schwartz function $f(u)$, then $\pi(n_x)v$ corresponds to $f(u-x)$ and  $\pi(a_y)v$ corresponds
 $f(\f uy)|y|^\mu\sgn(a)^\d$, where $\mu\in \C$ and $\d\in \Z/2\Z$ depend  on the archimedean type of the representation $\pi$.
The distribution $\tau$ is periodic, hence  tempered (meaning that it can be integrated against arbitrary Schwartz functions).  Its restriction to the real line has a Fourier series,
\begin{equation}\label{tauasfourierseries}
    \tau(u) \ \ = \ \ \sum_{n\,\neq\,0} c_n \, e^{2\pi i n u}\,,
\end{equation}
with coefficients that grow at most polynomially in $n$
(see \cite{microsoft}).
Thus
\begin{equation}\label{phiatauvfourier}
\aligned
    \phi(n_xa_y) \ \ & = \ \ |y|^\mu \,\sgn(y)^\d \, \int_{\R} \ \sum_{n\,\neq\,0} c_n \, e^{2\pi i n u} \,f(\smallf{u-x}y)\,du \\
    & = \ \ |y|^{\mu+1} \,\sgn(y)^\d \, \sum_{n\,\neq\,0} c_n \,e^{2\pi i n x}\,\int_\R e^{2\pi i n y u }\,f(u)\,du \\
    & = \ \
    |y|^{\mu+1} \,\sgn(y)^\d \, \sum_{n\,\neq\,0} c_n \,e^{2\pi i n x}\,\widehat{f}(-ny )\,.
\endaligned
\end{equation}
The estimate (\ref{mainestimate}) implies exponential decay (of the same caliber)
 for each Fourier coefficient of $\phi$, in particular the first coefficient.  In particular,
 it implies $\widehat{f}(-y)$ decays faster than the exponential of some fractional power of $y$.
  However, it is an arbitrary Schwartz function, and not all Schwartz functions have that decay
   property, e.g. $y\mapsto e^{-(\ln(y^2+1))^2}$.  We conclude that the exponential decay estimate (\ref{mainestimate}) does not
    hold for all smooth cusp forms.

\subsection{Exponential decay estimates in string theory}\label{stringtheory}

A number of interesting quantities in string theory are, by their very definition, automorphic functions.
This is because  split exceptional real groups in the $E_n$ series  arise as duality groups in toroidal
compactifications in Cremmer-Julia supergravity, and as a consequence the coupling constants of the theory
lie in a symmetric space $G/K$.  In type IIB string theory, $S$-, $T$-, and $U$-dualities force invariance
 under a large discrete subgroup $\G\subset G$ \cite{Hull:1994ys}. The purpose of this section is to discuss the exact exponent
  $d$ in the estimate (\ref{mainestimate}), and its string theory interpretation in various situations.  Much of the
  discussion here is an outgrowth of the second-named author's joint work with Michael Green and Pierre Vanhove \cite{GMV}, who
  we thank for their insightful comments on the questions posed here.  More information about these questions can be found
  in \cites{Bao:2009fg,GMV,gmrv,Pioline:2009qt,Persson:2011xi} and in related references therein.

\aspoint{{\bf Decay estimates for Fourier coefficients}}\label{decayofcoeffs}

In various particular problems supersymmetry provides extra information such as differential equations (e.g., for  the Laplace-Beltrami operator), suggesting that the automorphic functions are automorphic forms in these situations.  In addition, at times supersymmetry predicts the asymptotic behaviour of various Fourier coefficients of these automorphic forms.

Although the exponential decay estimate (\ref{mainestimate}) is stated for cuspidal automorphic forms $\phi$,
 it of course applies to any Fourier coefficient defined by integration of a character over a unipotent subgroup $U'\subset U$:
\begin{equation}\label{phichi}
    \phi_\chi(g) \ \ := \ \ \int_{(\G\cap U')\backslash U'}\phi(ug)\,\chi(u)^{-1}\,du\,,
\end{equation}
where $du$ is the Haar measure normalized to give the quotient $(\G\cap U')\backslash U'$ volume 1 and $\chi$ is a character of $U$ which is trivial on $\G\cap U'$.  Indeed, since the integration is over a compact quotient $\phi_\chi$ also obeys the estimate (\ref{mainestimate}).  Thus Theorem~\ref{mainthm} gives the exponential decay of these Fourier coefficients.
When $\phi$ is not cuspidal similar decay is expected for nontrivial characters $\chi$, since an automorphic form minus its constant terms
has rapid decay (this was essentially rederived in the analysis of (\ref{telescope}) above).

\aspoint{{\bf Some precise questions on the  exponential decay for maximal parabolics}}

Let $\a$ denote a positive simple root and $P=P_\a\supset B$ denote its associated standard maximal parabolic subgroup of $G$, which has the factorization $P_\a=A_\a M_\a U_\a$ in terms of its unipotent radical $U_\a$ and the center $A_\a$ of its reductive part $A_\a M_\a$.  For $a\in A_\a$ recall that  $a^\a \in \R$ is defined by the formula
$\Ad(a)X_\a = a^{\a} X_\a,$
where $X_\a$ is a root vector for the simple root $\a$.  We shall use this coordinate
to describe limits along the group $A_\a$, which often has an interpretation in string theory, e.g., of a compactifying radius or string coupling constant.  We do not assume that $\phi$ is a cusp form but we do assume that $\chi$ is nontrivial in order to exclude constant terms.

{\bf Question A:} Is it always that case that there exists constants $c_1=c_1(g,\phi,\chi)$, $c_2=c_2(g,\phi,\chi)$, and $d=d(\phi,\chi)$ such that
\begin{equation}\label{questionA}
    \phi_\chi(ag) \ \ \simeq \ \ c_1(g,\phi,\chi)\,\exp\(-c_2(g,\phi,\chi)\,(a^{\a})^{d(\phi,\chi)}\)
\end{equation}
for $g$ fixed, as $a^\a\rightarrow\infty$?  Here the symbol $\simeq$ is used to indicate asymptotic equality up to polynomials in the $a^\a$.   (Such factors do arise, e.g., in
\cite[(H.31)]{GMV}).

\vspace{.5cm}

{\bf Question B:}  Assuming an affirmative answer to Question A, does $d(\phi,
\chi)$ depend nontrivially on $\chi$, or is instead always 1 (as it is in all currently known cases)?

\vspace{.5cm}

\noindent
Let us now consider the lower central series
\begin{equation}\label{Ualphaandcommutators}
    U_\alpha \ \ = \ \ U^{(1)} \ \ \supset \ \ U^{(2)} \ \ \supset \ \ U^{(3)} \ \ \supset \ \ \cdots \ \ \supset \ \ U^{(\ell)} \ \ = \ \ \{e\}\,,
\end{equation}
where $U^{({k+1})}=[U^{(k)},U_\alpha]$.

\vspace{.5cm}

{\bf Question C:} Fix $1\le k \le \ell$ and a nontrivial character $\chi$ of $U_k$ which is  trivial on $\G\cap U^{(k)}.$  Assuming an affirmative answer to Question $A$, is $d(\phi,\chi)=k$?

For example, when $U_\a$ is a Heisenberg group Question C asks whether or not nontrivial Fourier coefficients along the one-dimensional center of $U_\a$ decay significantly faster than  nontrivial  Fourier coefficients on the full group $U_\a$ do.  We present some motivation for this question below.

\aspoint{{\bf Examples from string theory}}

We shall now describe the origin of question C in string theory.  Let $G$ denote the
split real form of the simply connected Chevalley group of type $E_n$, where $n=6$, $7$, or $8$.
 Let $\a$ denote the last positive simple root of $G$ in the usual Bourbaki numbering.  Thus $M_\a$ is a simply connected
 Chevalley group of type $Spin(5,5)$, $E_6$, or $E_7$ in these three cases.

 For particular Eisenstein series (which are not cuspidal, but for which our arguments
  nevertheless give  the same bounds on Fourier coefficients) the nonvanishing Fourier coefficients
   are expected to obey precise asymptotics as $a^\a$ goes to infinity.  This is because $a^{\a}$
    corresponds to the  decompactifying radius in the toroidal compactifaction.  When such a radius
    tends to infinity, the compactified theory resembles that of a higher dimensional theory compactified
     on a lower-dimensional torus, and has duality group $M_\a$. Precise  statements often give exponential
      decay in terms of an effective action, which in our context implies that $d(\phi,\chi)=1$ when $U_\a$
      is abelian and $\a$ comes from one
      the three terminal nodes of $G$'s Dynkin diagram.  This motivates Question B  and the $k=1$ case of Question C.

 The string theory interpretation is apparently different if $U_\a$ is nonabelian, for example a
  Heisenberg group.  In this case there is an additional charge to consider, which is instead
   consistent with $d(\phi,\chi)=2$ for the $k=2$ case of Question C.  Evidence for this is presented in \cite{Pioline:2009qt,Bao:2009fg,Persson:2011xi}.

 It should be stressed that string theory does not provide any additional insight into the decay rate of general automorphic forms, other than pointing to subtleties previously unrecognized in particular examples.

\end{document}